\documentclass{article}
\usepackage{graphicx}
\usepackage{amsmath}
\usepackage{amsthm}
\usepackage{amssymb}%
\usepackage[numbers,square]{natbib}

\newcommand{\E}{\mathbb E}
\newcommand{\R}{\mathbb{R}}
\newcommand{\N}{\mathbb{N}}

\newcommand{\Z}{\mathbb{Z}}

\renewcommand{\P}{\mathbb{P}}
\newcommand{\Var}{\mathop{\mathrm{Var}}\nolimits}

\newcommand{\esssup}{\mathop{\mathrm{esssup}}\nolimits}

\newcommand{\NN}{\mathcal{N}}
\newcommand{\FF}{\mathcal{F}}

\newcommand{\bbb}{\mathfrak{b}}

\newcommand{\eps}{\varepsilon}

\newcommand{\todistr}{\stackrel{\mathcal{D}}{\to}}
\newcommand{\toprobab}{\stackrel{\mathcal{P}}{\to}}
\newcommand{\nottoprobab}{\stackrel{\mathcal{P}}{\nrightarrow}}
\theoremstyle{plain}
\newtheorem{theorem}{Theorem}[]
\newtheorem{lemma}{Lemma}[]
\newtheorem{corollary}{Corollary}[]
\newtheorem{proposition}{Proposition}[]

\theoremstyle{definition}
\newtheorem{remark}{Remark}[]
\newtheorem{ass}{Assumption}[]

\theoremstyle{remark}

\begin{document}

\title{Limiting Distributions for Sums of Independent Random Products}
\author{Zakhar Kabluchko}
\maketitle
\begin{abstract}
Let $\{X_{i,j}: (i,j)\in\N^2\}$ be a two-dimensional array of independent copies of a random variable $X$, and let $\{N_n\}_{n\in\N}$ be a sequence of natural numbers such that $\lim_{n\to\infty}e^{-cn}N_n=1$ for some $c>0$. Our main object of interest is the sum of independent random products
$$
Z_n=\sum_{i=1}^{N_n} \prod_{j=1}^{n}e^{X_{i,j}}.
$$
It is shown that the limiting properties of $Z_n$, as $n\to\infty$, undergo phase transitions at two critical points $c=c_1$ and $c=c_2$. Namely, if $c>c_2$, then $Z_n$ satisfies the central limit theorem with the usual normalization, whereas for $c<c_2$, a totally skewed $\alpha$-stable law appears in the limit. Further, $Z_n/\E Z_n$ converges in probability to $1$ if and only if $c> c_1$.
If the random variable $X$ is Gaussian, we recover the results of Bovier, Kurkova, and L\"owe [Fluctuations of the free energy in the REM and the $p$-spin SK models. Ann.\ Probab.\ 30(2002), 605-651].
\end{abstract}

\section{Introduction and statement of results}\label{sec:intro}
\subsection{Introduction}\label{sec:intro_sub}
The aim of the present paper is to study the limiting distribution, as $n\to\infty$, of the random variable obtained by adding up a large number $N_n$ of independent summands, each summand being a product of $n$ independent positive-valued random variables.
To be more precise, let $\{X_{i,j}: (i,j)\in\N^2\}$ be a two-dimensional array of independent copies of a random variable $X$, and let $\{N_n\}_{n\in\N}$ be a sequence of natural numbers such that for some $c>0$,
\begin{equation}\label{eq:asympt_N}
\lim_{n\to\infty} e^{-cn} N_n=1.
\end{equation}
Then we are interested in the limiting properties, as $n\to\infty$, of the random variable 
\begin{equation}\label{eq:def_w}
Z_n=\sum_{i=1}^{N_n} \prod_{j=1}^{n} e^{X_{i,j}}. 
\end{equation}
Setting $Y_{i,j}:=e^{X_{i,j}}$, we may rewrite~\eqref{eq:def_w} as
$Z_n=\sum_{i=1}^{N_n} \prod_{j=1}^{n} Y_{i,j}$,
which leads to the ``sum of independent random products'' interpretation mentioned in the title of the paper.

To give a motivation for studying $Z_n$, let us consider the following model. Suppose we observe a large number of independent objects whose sizes evolve in time. Each object has size $1$ at time $0$ and  grows (or decays) in a random multiplicative way. This means that the size of the $i$-th object at time $j$ is obtained by multiplying the size of the same object at time $j-1$ by some positive random variable $Y_{i,j}$, where the variables $\{Y_{i,j}: (i,j)\in\N^2\}$ are supposed to be i.i.d. Then the size of the $i$-th object at time $n$ is given by $\prod_{j=1}^{n} Y_{i,j}$. Therefore, $Z_n$ may be interpreted as the total size of $N_n$ independently evolving objects at time $n$. Of course, the words ``object'' and ``size'' have to be understood in a very general sense. For example, one may consider independent particles whose mass, charge, or energy changes randomly and multiplicatively due to the motion in a random environment. Our limit theorems for $Z_n$ will be applicable if the number $N_n$ of objects is much larger  than the time $n$, cf.~\eqref{eq:asympt_N}.

A closely related model is the sum of independent random exponentials
\begin{equation}\label{eq:def_S_N}
S_N(t)=\sum_{i=1}^{N} e^{\sqrt {t} X_i},
\end{equation}
where  $\{X_i, i\in \N\}$ is  a sequence of independent copies of a random variable $X$. The limiting properties of $S_N(t)$ were studied by a number of authors, mostly in connection with Derrida's Random Energy Model, a simple model of disordered systems~\cite{bovier_book06}.
Weak and strong laws of large numbers for $S_N(t)$ as $N,t\to\infty$ were proved in~\cite{eisele83}, \cite{olivieri_picco84}, \cite{galves_etal89}.
The study of limiting distributions of $S_N(t)$ began with \citet{schlather01}, who considered quantities essentially equivalent to $S_N(t)$ to obtain a continuous interpolation between the central limit theorem and limit theorems for extreme values. To describe his idea, let us consider two ``boundary'' cases. In the first case, assume that $N\to\infty$, whereas $t$ remains constant. Then the limiting distribution of $S_N(t)$ is Gaussian by the central limit theorem. In the second case, assume that $t\to\infty$, whereas  $N$ is some fixed, but very large constant. Then the sum $S_N(t)$ reduces essentially to the largest summand and hence, one may expect that in the second case, the distribution of $S_N(t)$ will be governed by the extreme value theory. Now, \citet{schlather01} considered the \textit{intermediate} case in which both $t$ and $N$ tend to $\infty$ in some \textit{synchronized} way, and obtained (depending on the law of $X$) several non-trivial families of limiting distributions for $S_N(t)$ interpolating between the boundary cases.
Independently of~\cite{schlather01}, the limiting distributions of  $S_N(t)$ were studied by \citet{bovier_etal02} in the case of standard Gaussian $X$.
The authors of~\cite{bovier_etal02} were motivated by the Random Energy Model and obtained a family of limiting laws similar to the families found in~\cite{schlather01}.
The results of~\cite{schlather01} and~\cite{bovier_etal02} were largely extended in~\cite{ben_arous_etal05} (where, additionally, a new method was introduced and the structure of the family of limiting laws was clarified), \cite{bogachev06}, \cite{bogachev07}, \cite{janssen09}.  Let us also mention that a picture of limiting laws similar to that found in the papers cited above was obtained in a different context in~\cite{csorgo_etal86}, \cite{csorgo_mason86}.

Our aim is to determine the structure of the family of limiting laws for the sum of independent random products $Z_n$.
If the distribution of $X$ is Gaussian, then $Z_n$ reduces essentially to the sum of random exponentials $S_N(t)$. In this case, we will recover the results of~\cite{bovier_etal02}.



\begin{remark}
After the first version of this paper was submitted to the arXiv, it has been pointed out to the author by Leonid Bogachev that the results of the paper (except for Theorem~\ref{theo:stable_lattice}) have been obtained by M.\ Cranston and S.\ Molchanov, Limit laws for sums of products of exponentials of iid random variables, Isr.\ J.\ Math., 148 (2005), 115--136 and O.\ Khorunzhiy, Limit theorem for sums of products of random variables, Markov Process.\ Related Fields 9 (2003), 675--686.

\end{remark}

\subsection{Notation}\label{sec:not_0}
To state our results, we need to recall some facts connected with Cram\'er's large deviations theorem, see e.g.~\cite[\S~2.2]{dembo_book93}. Let $X$ be a random variable (always  assumed to be \textit{non-degenerate}) such that
\begin{equation}\label{eq:def_varphi}
\varphi(t):=\log \E[e^{tX}] \text{ is finite for all } t\geq 0.
\end{equation}
The function $\varphi$ vanishes at $0$, it is continuous on $[0,\infty)$, infinitely differentiable on $(0,\infty)$, and strictly convex.  Define
\begin{equation}\label{eq:def_beta_0}
\beta_0:=\lim_{t\to 0+}\varphi'(t)=\E X,\;\;\; \beta_{\infty}:=\lim_{t\to +\infty} \varphi'(t)=\esssup X.
\end{equation}
It may happen that $\beta_0=-\infty$ or $\beta_{\infty}=+\infty$.

Let  $I:(\beta_0,+\infty)\to (0,+\infty]$  be the Legendre-Fenchel transform of $\varphi$ defined  by
\begin{equation}\label{eq:def_I}
I(\beta)=\sup_{t\geq 0} (\beta t-\varphi(t)).
\end{equation}
The function $I$ has following properties, see~\cite[\S~2.2]{dembo_book93}. It is finite on $(\beta_0,\beta_{\infty})$, and we have $I(\beta)=+\infty$ for $\beta>\beta_{\infty}$. On the interval $(\beta_0,\beta_{\infty})$, the function $I$ is strictly increasing, strictly convex, and infinitely differentiable. If $\beta_{\infty}\neq +\infty$, then  $I(\beta_{\infty})=\lim_{\beta\uparrow \beta_{\infty}}I(\beta)$, and the value $I(\beta_{\infty})$ may be both finite and infinite.   Let $\beta=\varphi'(\alpha)$ for some $\alpha>0$ (note that this implies that $\beta\in (\beta_0,\beta_{\infty})$). Then the supremum in~\eqref{eq:def_I} is attained at $t=\alpha$ and hence,
\begin{equation}\label{eq:I_varphi_prime}
I(\varphi'(\alpha))=\alpha\varphi'(\alpha)-\varphi(\alpha),\;\;\; \alpha>0.
\end{equation}

We will see that the limiting properties of the sum of independent random products $Z_n$ undergo  phase transitions at two ``critical points'' $c=c_1$ and $c=c_2$ given by
\begin{equation}\label{eq:def_c1_c2}
c_1:=I(\varphi'(1))=\varphi'(1)-\varphi(1),\;\;\; c_2:=I(\varphi'(2))=2\varphi'(2)-\varphi(2).
\end{equation}
The function $\alpha\mapsto I(\varphi'(\alpha))$, $\alpha>0$, is strictly increasing. Therefore, $0<c_1<c_2<\infty$.

\subsection{Results on limiting distributions}\label{sec:res_lim_distr}
Our first theorem shows that for $c>c_2$, the sum $Z_n$ satisfies the central limit theorem.
\begin{theorem}\label{theo:normal}
Suppose that~\eqref{eq:asympt_N} and~\eqref{eq:def_varphi} hold, and assume that $c>c_2$. Then
$$
\frac{Z_n-\E Z_n}{\sqrt{\Var Z_n}}\todistr \NN(0,1),\;\;\; n\to\infty.
$$
\end{theorem}

At the critical point $c=c_2$ the limiting distribution is still Gaussian, but with variance $1/2$.
\begin{theorem}\label{theo:crit}
Suppose that~\eqref{eq:asympt_N} and~\eqref{eq:def_varphi} hold, and assume that $c=c_2$. Then
$$
\frac{Z_n-\E Z_n}{\sqrt{\Var Z_n}}\todistr \NN(0,1/2),\;\;\; n\to\infty.
$$
\end{theorem}

The next theorem shows that for $c<c_2$, the central limit theorem breaks down and instead, a totally skewed $\alpha$-stable distribution appears in the limit. To state it, we need to assume that the distribution of the random variable $X$ is \textit{non-lattice}.  Recall that a random variable $X$ is called \textit{lattice} if there exist $h,a\in\R$ such that the values of $X$ are a.s.\ of the form $hn+a$, $n\in\Z$.
\begin{theorem}\label{theo:stable}
Suppose that~\eqref{eq:asympt_N} and~\eqref{eq:def_varphi} hold, $c\in(0,c_2)$, and assume that the distribution of $X$ is non-lattice.  Define $\alpha\in (0,2)$ as the unique solution of the equation $I(\varphi'(\alpha))=c$, and let $\beta=\varphi'(\alpha)$. Set
\begin{align}
A_n
&=
\begin{cases}
\E Z_n=N_n e^{\varphi(1)n}, & \text{ if }c\in (c_1,c_2),\\
N_n \E[e^{S_{n}}1_{S_{n}\leq b_n}], &\text{ if }    c=c_1, \text{ where } S_{n}=\sum_{j=1}^n X_{1,j},\\
0,     &\text{ if }c\in (0,c_1),
\end{cases}\label{eq:def_An}\\
B_n
&=e^{b_n}, \text{ where } b_n=\beta n -\alpha^{-1}\log \left( \alpha \sqrt{2\pi\varphi''(\alpha) n} \right).
\label{eq:def_Bn_large}
\end{align}
Then
$$
\frac{Z_n-A_n}{B_n}\todistr \FF_{\alpha}, \;\;\; n\to\infty,
$$
where $\FF_{\alpha}$ is an $\alpha$-stable distribution with skewness parameter $+1$. The characteristic function $\phi_{\alpha}$ of $\FF_{\alpha}$ is given by 
$$
\log \phi_{\alpha}(u)=
\begin{cases}
-\Gamma(1-\alpha)|u|^{\alpha}\exp\left( -i\frac{\pi}2 \alpha \operatorname{sgn} (u) \right), &\text{ if }\alpha\neq 1,\\
iu(1-\gamma)-|u|\left(\frac {\pi}{2}+ i\operatorname{sgn} (u)\log|u|\right), &\text{ if }\alpha=1,
\end{cases}
$$
where $\gamma$ is the Euler constant.
\end{theorem}
\begin{remark}
To see that the equation $I(\varphi'(\alpha))=c$ has a unique solution $\alpha$ for every $c\in(0,c_2)$, note that the function $\alpha\mapsto I(\varphi'(\alpha))$, $\alpha>0$, is continuous and strictly increasing, and that $\lim_{\alpha\to 0+}I(\varphi'(\alpha))=0$. For the latter fact see e.g.\ Eq.~2.2.8 in~\cite{dembo_book93}.
\end{remark}
\begin{remark}\label{rem:trunc}
It will be shown in Section~\ref{sec:trunc_clt} that in the case $c=c_1$, the sequence $A_n$ satisfies, as $n\to\infty$,
\begin{equation}\label{eq:asympt_An}
A_n=\frac 12 N_ne^{\varphi(1) n}
\left\{
1 -\frac{1}{\sqrt {2\pi \varphi''(1) n}}
\left( \log (2\pi\varphi''(1) n)-\frac {\varphi'''(1)}{3\varphi''(1)}\right)+o\left( \frac{1}{\sqrt n}\right)
\right\}.
\end{equation}
The right-hand side of~\eqref{eq:asympt_An} provides an alternative way to choose $A_n$ in the case $c=c_1$.
\end{remark}

In the next theorem we describe what happens if the distribution of $X$ is lattice. As we will see, in this case there is no convergence to an $\alpha$-stable distribution. Instead, a family of infinitely divisible distributions with discrete L\'evy measures appears as the set of accumulation points of the appropriately normalized sequence $Z_n$. Since $Z_n$ changes only by a constant factor $e^{-na}$ if we replace $X$ by $X-a$, there is no restriction of generality in making the following assumption.
\begin{ass}\label{ass:non_lattice}
There is $h>0$ such that the values of $X$ belong with probability $1$ to the lattice $h\Z=\{hn: n\in\Z\}$, and, moreover, $h$ is the largest number with this property.
\end{ass}
It will be convenient to use the following notation: for $b\in \R$ we set
$$
[b]_h:=\max\{a\in h\Z: a\leq b\},\;\;\; \{b\}_h:=b-[b]_h\in [0,h).
$$
\begin{theorem}\label{theo:stable_lattice}
Suppose that~\eqref{eq:asympt_N} and~\eqref{eq:def_varphi} hold, $c\in(0,c_2)$, and let  Assumption~\ref{ass:non_lattice} be satisfied.
Define $\alpha\in (0,2)$ as the unique solution of the equation $I(\varphi'(\alpha))=c$, and let  $\beta=\varphi'(\alpha)$. Define $A_n$ as in~\eqref{eq:def_An} and set
\begin{equation}\label{eq:def_bn_lattice}
B_n=e^{b_n}, \text{ where } b_n=\beta n - \alpha^{-1}\log\left(h^{-1}\sqrt{2\pi \varphi''(\alpha)n}\right).
\end{equation}
Define $\Delta_n=\{b_n\}_h$ and let $\{n_k\}_{k\in\N}$ be an increasing integer sequence such that $\Delta:=\lim_{k\to\infty}\Delta_{n_k}$ exists.
Then
$$
\frac{Z_{n_k}-A_{n_k}}{B_ {n_k}}\todistr \FF_{\alpha, \Delta}, \;\;\; k\to\infty.
$$
Here, $\FF_{\alpha, \Delta}$ is an infinitely divisible distribution whose characteristic function $\phi_{\alpha,\Delta}$ has a L\'evy-Khintchine representation
$$
\log \phi_{\alpha,\Delta}(u)
=
iC_{\alpha, \Delta}u + \sum_{ x\in  \exp(h\Z-\Delta) }\left(e^{iux}-1- iux 1_{x<1}\right) x^{-\alpha},
$$ 
where $\exp(h\Z-\Delta)$ denotes the set $\{e^{hn-\Delta}: n\in\Z\}$, and $C_{\alpha,\Delta}$ is some constant.
\end{theorem}
\begin{corollary}\label{cor:tight}
Under the assumptions of Theorem~\ref{theo:stable_lattice},
$$
\left \{\frac{Z_n-A_n}{B_n}, n\in\N \right\}
$$
is a tight sequence of random variables.
\end{corollary}

\subsection{Weak laws of large numbers}\label{sec:weak_lln}
In this section we state our results on the stochastic convergence of $Z_n$ and the associated ``finite-scale free energy''  $(1/n)\log Z_n$.  We denote by $\toprobab$ convergence in probability.
\begin{theorem}\label{theo:weak_lln}
Suppose that~\eqref{eq:asympt_N} and~\eqref{eq:def_varphi} hold.
Then, as $n\to\infty$,
\begin{equation}\label{eq:weak_lln}
e^{-(\varphi(1)+c)n} Z_n\toprobab
\begin{cases}
1, & \text{ if } c>c_1,\\
1/2, &\text{ if } c=c_1.
\end{cases}
\end{equation}
\end{theorem}
\begin{remark}\label{rem:no_seq}
It will be shown in Section~\ref{sec:proof_rem} that in the case $c<c_1$ there is no sequence $a_n$ with $Z_n/a_n\toprobab 1$ as $n\to\infty$.
\end{remark}

In the next theorem we compute the crude logarithmic asymptotics of $Z_n$. If the distribution of $X$ is Gaussian, we recover a result of~\cite{eisele83}, see also~\cite{olivieri_picco84}, on the free energy in the Random Energy Model.
\begin{theorem}\label{theo:free_energy}
Suppose that~\eqref{eq:asympt_N} and~\eqref{eq:def_varphi} hold. Then, as $n\to\infty$,
\begin{equation}\label{eq:free_energy}
\frac 1 n \log Z_n \toprobab
\begin{cases}
\varphi(1)+c, &\text{ if } c\in [c_1,\infty),\\
I^{-1}(c), &\text{ if }c\in (0, c_1],
\end{cases}
\end{equation}
where $I^{-1}$ denotes the inverse function of $I$.
\end{theorem}

It was pointed out in~\cite{eisele83} that an analogue of Theorem~\ref{theo:free_energy} for the sum of random exponentials $S_N(t)$ can be seen as a third-order phase transition in the Random Energy Model. The next proposition shows that the third-order character of the phase transition is preserved if $S_N(t)$ is replaced by $Z_n$.
\begin{proposition}\label{prop:third_order}
Let $f_-(c)=\varphi(1)+c$ and $f_+(c)=I^{-1}(c)$. Then
$f_-(c_1)=f_+(c_1)$ and $f_-'(c_1)=f_+'(c_1)$, but $f_-''(c_1)\neq f_+''(c_1)$.
\end{proposition}

\subsection{Example: Uniform stick breaking}
Let us compute the critical points $c_1$ and $c_2$ explicitly assuming that $Y:=e^X$ is uniformly distributed on the interval $[0,1]$. To give a motivation, imagine a large number $N_n$ of sticks of unit length and mass.
Suppose that each stick is broken at a uniformly distributed point, resulting in two fragments. One of the fragments, say the right one, is ignored, whereas the left one is again broken at a uniformly distributed point, and the procedure is repeated $n$ times for each of the $N_n$ sticks. The total mass of the $N_n$ (very small) fragments obtained in this way after $n$ steps is given by $Z_n$  defined in~\eqref{eq:def_w} with $e^X$ being uniformly distributed on $[0,1]$.
To compute $c_1$ and $c_2$, note that
$$
\varphi(t)=\log \E[e^{tX}]=\log \E[Y^t]=\log \int_{0}^1 y^{t}dy=-\log (1+t).
$$
Hence, $\varphi'(t)=-1/(1+t)$. It follows from~\eqref{eq:def_c1_c2} that the critical points $c_1$ and $c_2$ are given by
$$
c_1=\log 2-\frac 12\approx 0.19,\;\;\; c_2=\log 3-\frac 23\approx 0.43.
$$
An easy computation shows that $\beta_0=-1$, $\beta_{\infty}=0$, and $I(\beta)=-(1+\beta)-\log (-\beta)$ for $\beta\in (-1,0)$. For example, Theorem~\ref{theo:free_energy} yields, as $n\to\infty$,
$$
\frac 1 n \log Z_n \toprobab
\begin{cases}
c-\log 2, &\text{ if } c\in [c_1,\infty),\\
-\gamma(c), &\text{ if }c\in (0, c_1],
\end{cases}
$$
where $\gamma=\gamma(c)$ is defined as the unique solution of the equation $\gamma-\log\gamma=1+c$ in the interval $(0,1)$.

Of course, similar explicit calculations can be done  in a number of further special cases, e.g.\ if $Y$ is $B$- or $\Gamma$-distributed.


\subsection{Heuristic arguments}
Some of the above results admit a natural non-rigorous interpretation.  First, we give a non-rigorous argument justifying Theorem~\ref{theo:free_energy}. There are two natural ways to try to guess the limit of $(1/n)\log Z_n$ as $n\to\infty$. The first way is to assume that $Z_n$ is well approximated by its expectation. This leads to the approximation
$$
\frac 1n \log  Z_n\approx \frac 1n \log \E Z_n \approx \frac 1n \log \left(N_n e^{\varphi(1)n}\right) \approx  \varphi(1)+c.
$$
By Theorem~\ref{theo:free_energy}, this gives a correct result provided that $c\in[c_1,\infty)$. The second way is to assume that $Z_n$ is dominated by the maximal summand in~\eqref{eq:def_w}. It can be shown that
\begin{equation}\label{eq:erd_renyi}
\frac 1n \max_{i=1,\ldots, N_n} S_{i,n} \toprobab I^{-1}(c), \;\;\; n\to\infty,
\end{equation}
where $S_{i,n}=\sum_{j=1}^n X_{i,j}$. This leads to the approximation
$$
\frac 1n \log Z_n \approx \frac 1n \log \left(\max_{i=1,\ldots,N_n}e^{S_{i,n}}\right) \approx \frac 1n \log \left( e^{I^{-1}(c)n} \right)=I^{-1}(c).
$$
By Theorem~\ref{theo:free_energy}, this gives a correct result provided that $c\in(0,c_1]$.

Let us return to~\eqref{eq:erd_renyi}. It may be thought of as a simplified (decoupled) version of the Erd\"os-R\'enyi law of large numbers, see e.g.~\cite[Theorem~2.4.3]{csorgo_book81}. A statement which is more precise than~\eqref{eq:erd_renyi} was proved in Theorem~2 of~\cite{ivchenko73} (see also~\cite{durrett79}) where it was shown that the left-hand side of~\eqref{eq:erd_renyi} has limiting Gumbel distribution.  Moreover, it was shown there that the upper order statistics of the sequence $\{S_{i,n}\}_{i=1}^{N_n}$ can be approximated, as $n\to\infty$, by the point process $\{\alpha^{-1} U_i+b_n\}_{i=1}^{\infty}$, where $\{U_i\}_{i=1}^{\infty}$ is a Poisson point process on $\R$ having intensity $e^{-u}du$, and $\alpha$ and $b_n$ are as in Theorem~\ref{theo:stable}.  Hence, we obtain an approximation
$$
e^{-b_n} Z_n \approx  \sum_{i=1}^{\infty} e^{\alpha^{-1} U_i}, \;\;\; n\to\infty.
$$
At least for $\alpha\in (0,1)$, the right-hand side converges a.s.\ and has an $\alpha$-stable distribution with skewness parameter $+1$. This gives a non-rigorous justification of Theorem~\ref{theo:stable} in the case $\alpha\in (0,1)$. This argument is close to the approach used in~\cite{schlather01} and~\cite{bovier_etal02}, and could be turned into a rigorous proof. However, we prefer to proceed in a different way. Our proof of Theorem~\ref{theo:stable} will be based on two ingredients. The first ingredient is the classical summation theory for triangular arrays of independent random variables~\cite{gnedenko_book54}. This approach is rather natural in view of the triangular array structure of~\eqref{eq:def_w}, and was suggested in~\cite{ben_arous_etal05}. The second ingredient is a precise version of Cram\'er's large deviations theorem which will be recalled in Section~\ref{sec:ld}.

\subsection{Organization of the paper}
Theorems~\ref{theo:normal}, \ref{theo:crit}, \ref{theo:stable}, \ref{theo:stable_lattice} will be proved in Sections~\ref{sec:proof_normal}, \ref{sec:proof_crit}, \ref{sec:proof_stable}, \ref{sec:proof_stable_lat}, respectively. The weak laws of large numbers of Section~\ref{sec:weak_lln} will be then derived in Section~\ref{sec:proof_lln} as consequences of the results of Section~\ref{sec:res_lim_distr}.

\section{Facts about precise large deviations}\label{sec:ld}
Throughout the rest of the paper let $\{X_i: i\in\N\}$ be i.i.d.\ copies of a random variable $X$ satisfying~\eqref{eq:def_varphi}, and denote by $S_n=X_1+\ldots+X_n$ their partial sums.
The following theorem on the precise asymptotic behavior of large deviation probabilities for sums of i.i.d\ variables due to~\cite{bahadur_rao60} and~\cite{petrov65} will play a crucial role in the sequel.
\begin{theorem}\label{theo:ld}
Suppose that~\eqref{eq:def_varphi} is satisfied.
Let $\beta\in(\beta_0,\beta_{\infty})$ and define $\alpha>0$ to be the unique solution of the equation $\varphi'(\alpha)=\beta$.
\begin{enumerate}
\item \label{p:1_th_petrov}
Assume that the distribution of $X$ is non-lattice.  Then
$$
\P[S_n\geq n\beta]\sim \frac{1}{\alpha \sqrt{2\pi \varphi''(\alpha)n}} e^{-n I(\beta)},
\;\;\; n\to\infty.
$$
\item \label{p:2_th_petrov}
Assume that the distribution of $X$ is lattice, and that Assumption~\ref{ass:non_lattice} is fulfilled for some  $h>0$. Then the following holds if $\beta\in n^{-1}h\Z$:
$$
\P[S_n = n\beta]\sim \frac{h}{\sqrt{2\pi \varphi''(\alpha)n}}   e^{-n I(\beta)},
\;\;\; n\to\infty.
$$
\end{enumerate}
Both statements hold uniformly in $\beta\in K$ for any compact set $K\subset (\beta_0,\beta_{\infty})$.
\end{theorem}
The next lemma holds both in the lattice and in the non-lattice case.
\begin{lemma}\label{cor:ld}
Suppose that~\eqref{eq:def_varphi} is satisfied.
Let $K$ be a compact subset of $(\beta_0, \beta_{\infty})$. Then there is a constant $C$ depending on $K$ such that for all $\beta\in K$ and for all $n\in\N$,
$$
\P[S_n\geq n\beta]\leq C n^{-1/2} e^{-n I(\beta)}.
$$
\end{lemma}
\begin{proof}
In the non-lattice case the lemma follows immediately from Part~\ref{p:1_th_petrov} of Theorem~\ref{theo:ld}. In the lattice case, we assume that there exist $a\in\R$ and $h>0$ such that the values of $X$ belong a.s.\ to the lattice  $a+h\Z$, and that $h$ is maximal with this property. It was proved in Part~II of Theorem~6 in~\cite{petrov65} that the following asymptotic relation holds uniformly in $\beta\in K$ as long as the values of $\beta$ are restricted to the lattice $a+hn^{-1}\Z$:
\begin{equation}\label{eq:petrov_latt}
\P[S_n \geq  n\beta]\sim \frac{1}{(1-e^{-\alpha h})}\frac{h}{\sqrt{2\pi \varphi''(\alpha)n}}   e^{-n I(\beta)},
\;\;\; n\to\infty.
\end{equation}
For general $\beta\in K$ set
$$
\beta_n=\max\{x\in a+hn^{-1}\Z : x\leq \beta\}.
$$
Note that $\beta_n>\beta-hn^{-1}$. Then, by convexity of $I$,
\begin{equation}\label{eq:ineq_122}
I(\beta_n)\geq I(\beta)+I'(\beta)(\beta_n-\beta)\geq I(\beta)-hn^{-1}I'(\beta)\geq I(\beta)-mn^{-1},
\end{equation}
where $m$ is a constant which is uniform in $\beta\in K$. We have, evidently,
\begin{equation}\label{eq:ineq_177}
\P[S_n \geq  n\beta] \leq \P[S_n \geq  n\beta_n].
\end{equation}
Applying to the right-hand side of~\eqref{eq:ineq_177} first~\eqref{eq:petrov_latt} and then~\eqref{eq:ineq_122}, we obtain
$$
\P[S_n \geq  n\beta]\leq C_1 n^{-1/2} e^{-n I(\beta_n)}
\leq C_1 n^{-1/2} e^{-n (I(\beta)-mn^{-1})}
\leq C_2 n^{-1/2} e^{-n I(\beta)}.
$$
This completes the proof in the lattice case.
\end{proof}
The next two lemmas are standard.
\begin{lemma}\label{lem:ld_est0}
For every $\beta> \beta_0$ and every $n\in\N$,
\begin{equation}\label{eq:ineq_markov}
\P[S_n\geq n\beta]\leq e^{-nI(\beta)}.
\end{equation}
\end{lemma}
\begin{proof}
For every fixed $t\geq 0$, the Markov inequality yields
$$
\P[S_n\geq n\beta]\leq e^{-n \beta t} \E[e^{tS_n}] = e^{-n(\beta t- \varphi(t))}.
$$
Recalling that by~\eqref{eq:def_I},  $I(\beta)=\sup_{t\geq 0} (\beta t- \varphi(t))$, we arrive at~\eqref{eq:ineq_markov}.
\end{proof}

\begin{lemma}\label{lem:I_prime}
For every $\alpha>0$, we have $I'(\varphi'(\alpha))=\alpha$.
\end{lemma}
\begin{proof}
Let $\beta\in (\beta_0,\beta_{\infty})$. Denote by $\theta(\beta)$  the unique solution of the equation $\varphi'(\theta(\beta))=\beta$. Then the supremum in~\eqref{eq:def_I} is attained at $t=\theta(\beta)$ and hence,
$$
I(\beta)=\beta \theta(\beta)-\varphi(\theta(\beta)).
$$
We will show that $I'(\beta)=\theta(\beta)$.
We have
$$
I'(\beta)=(b\theta(b)-\varphi(\theta(b)))'|_{b=\beta}=\theta(\beta)+\beta\theta'(\beta)-\varphi'(\theta(\beta))\theta'(\beta)=\theta(\beta).
$$
Setting $\beta=\varphi'(\alpha)$ and taking into account that $\varphi'(\theta(\alpha))=\alpha$, we obtain the statement of the lemma.
\end{proof}

The next theorem gives a first term in the asymptotic expansion in the central limit theorem, see Theorem~5.22 of~\cite{petrov_book95}. The only place where we will use it is the proof of Remark~\ref{rem:trunc} in Section~\ref{sec:trunc_clt}.
\begin{theorem}\label{theo:asympt_clt}
Let $\{X_i, i\in\N\}$ be i.i.d.\ zero-mean random variables having a non-lattice distribution and finite third moment. Let $\mu_2=\E X_1^2$ and $\mu_3=\E X_1^3$.
Then, uniformly in $x\in\R$,
$$
\Phi_n(x):=\P\left[\frac{S_n}{\sqrt{\mu_2 n}} \leq x\right]
= \Phi(x)+\frac{\mu_3(1-x^2)e^{-x^2/2}}{6\sqrt{2\pi n} \mu_2^{3/2}}+o\left(\frac{1}{\sqrt n}\right),\;\;\; n\to\infty.
$$
Here, $\Phi(x)=(2\pi)^{-1/2}\int_{-\infty}^x e^{-t^2/2}dt$ is the standard normal distribution function.
\end{theorem}

\section{Proof of Theorem~\ref{theo:normal}}\label{sec:proof_normal}
The proof will be based on the classical summation theory for triangular arrays of independent random variables, a method suggested in~\cite{ben_arous_etal05}. Recall that $\{X_i: i\in\N\}$ are i.i.d.\ copies of a random variable $X$ satisfying~\eqref{eq:def_varphi}, and that $S_n=X_1+\ldots+X_n$. For $n\in\N$ define a positive-valued random variable $W_n$ by
\begin{equation}\label{eq:def_bn}
W_n= e^{S_n-b_n}, \text{ where }
b_n =\frac{\log N_n+\varphi(2)n}{2}.
\end{equation}
Let $W_{n,1},\ldots, W_{n,N_n}$ be independent copies of the random variable $W_n$. Then $Z_n$ has the same distribution as $e^{b_n}\sum_{i=1}^{N_n}W_{n,i}$, and we can restate Theorem~\ref{theo:normal} in the following equivalent form:
\begin{equation}\label{eq:theo_normal_eq}
\sum_{i=1}^{N_n}W_{n,i}- N_n\E[W_n]\todistr\NN(0,1), \;\;\; n\to\infty.
\end{equation}
To see that~\eqref{eq:theo_normal_eq} is indeed equivalent to Theorem~\ref{theo:normal}, we need to show that $\Var Z_n \sim e^{2 b_n}$  as $n\to\infty$.
This can be done as follows:
$$
\Var Z_n
=N_n(\E[e^{2S_n}]-\E[e^{S_n}]^2)
=N_n(e^{\varphi(2)n}-e^{2\varphi(1)n})
\sim N_n e^{\varphi(2)n}
= e^{2b_n},
$$
where we have used that by strict convexity of $\varphi$, $\varphi(2)>2\varphi(1)$.

From now on, we concentrate on proving~\eqref{eq:theo_normal_eq}.  According to Theorem~4.3 on p.~119 of~\cite{petrov_book95}, the convergence in~\eqref{eq:theo_normal_eq} will be established after we have verified that the following three conditions hold for every $\tau>0$:
\begin{align}
&\lim_{n\to\infty} N_n\P[W_n>\tau]=0, \label{eq:clt_cond1}\\
&\lim_{n\to\infty} N_n \Var[W_n 1_{W_n\leq \tau}]=1, \label{eq:clt_cond2}\\
&\lim_{n\to\infty} N_n \E[W_n 1_{W_n> \tau}]=0. \label{eq:clt_cond3}
\end{align}

\subsection{Proof of~\eqref{eq:clt_cond1}}
Define $\beta_n=(b_n+\log \tau)/n$. Recalling~\eqref{eq:def_bn} and using Lemma~\ref{lem:ld_est0}, we obtain
\begin{equation}\label{eq:lim_beta_n}
N_n\P[W_n>\tau]=N_n\P[e^{S_n-b_n}>\tau]=N_n\P[S_n>n\beta_n]\leq N_n e^{-I(\beta_n)n}.
\end{equation}
Recall that by~\eqref{eq:asympt_N}, $N_n\sim e^{cn}$ as $n\to\infty$. To prove~\eqref{eq:clt_cond1}, it suffices to show that for some $\eps>0$,
$
I(\beta_n)>c+\eps
$
provided that $n$ is large.
We have
$$
\lim_{n\to\infty}\beta_n
=\frac 12 (c+\varphi(2))=\frac 12 (c_2+\varphi(2))+\frac 12 (c-c_2)=\varphi'(2)+\frac 12 (c-c_2),
$$
where the last equality follows from~\eqref{eq:def_c1_c2}.

Note that $I(\varphi'(2))=c_2$ by~\eqref{eq:def_c1_c2}, and $I'(\varphi'(2))=2$ by Lemma~\ref{lem:I_prime}. For every $a>0$ we have
\begin{equation}\label{eq:ineq_348}
I(\varphi'(2)+a) > c_2+2a.
\end{equation}
Indeed, if $\varphi'(2)+a>\beta_\infty$, then the left-hand side of~\eqref{eq:ineq_348} is infinite, whereas the right-hand side is finite, so that~\eqref{eq:ineq_348} holds. If $\varphi'(2)+a\leq \beta_\infty$, then~\eqref{eq:ineq_348} follows from the strict convexity of $I$ on $(\beta_0,\beta_\infty]$.  By lower semi-continuity of $I$, and by~\eqref{eq:ineq_348}, we have
\begin{align*}
\liminf_{n\to\infty} I(\beta_n)
\geq I\left(\varphi'(2)+\frac 12 (c-c_2)\right)>c_2+(c-c_2)=c.
\end{align*}
Thus, there is $\eps>0$ such that $I(\beta_n)>c+\eps$ provided that $n$ is sufficiently large.
Therefore, the right-hand side of~\eqref{eq:lim_beta_n} converges to $0$ as $n\to\infty$. This proves~\eqref{eq:clt_cond1}.

\subsection{Proof of~\eqref{eq:clt_cond2}}
First note that by~\eqref{eq:def_bn},
\begin{equation}\label{eq:ewn2}
\E[W_n^2]=\E[e^{2(S_n-b_n)}]=e^{\varphi(2)n-2b_n}=1/N_n.
\end{equation}
In view of this, in order to prove~\eqref{eq:clt_cond2}, it suffices to establish the following two equalities:
\begin{align}
\lim_{n\to\infty} N_n \E[W_n^2 1_{W_n > \tau}]&=0, \label{eq:clt_cond2a}\\
\lim_{n\to\infty} N_n^{1/2} \E[W_n 1_{W_n\leq \tau}]&=0. \label{eq:clt_cond2b}
\end{align}
Let us prove~\eqref{eq:clt_cond2b} first. We have
$$
N_n^{1/2} \E[W_n 1_{W_n\leq \tau}]\leq N_n^{1/2}\E[W_n]=N_n^{1/2}e^{-b_n}\E[e^{S_n}]=e^{(\varphi(1)-\varphi(2)/2) n}.
$$
By strict convexity of $\varphi$, we have $\varphi(1)<\varphi(2)/2$. This proves~\eqref{eq:clt_cond2b}.
Let us prove~\eqref{eq:clt_cond2a}. By Lemma~4.2 of~\cite{ben_arous_etal05}, we have for every $\delta>0$,
$$
\E[W_n^2 1_{W_n>\tau}]< \tau^{-2\delta} \E[W_n^{2+2\delta}].
$$
Thus, it suffices to show that for some $\delta>0$,
\begin{equation}\label{eq:lim_n_E_Wn_delta}
\lim_{n\to\infty}N_n\E[W_n^{2+2\delta}]=0.
\end{equation}
We have
\begin{equation}\label{eq:est1}
\E[W_n^{2+2\delta}]
= e^{-(2+2\delta)b_n} \E[e^{(2+2\delta)S_n}]
= e^{-c_n(\delta) n},
\end{equation}
where $c_n(\delta)=(2+2\delta)(b_n/n)-\varphi(2+2\delta)$.
Recalling~\eqref{eq:def_bn}, we obtain
$$
\lim_{n\to\infty} c_n(\delta)=(1+\delta)(c+\varphi(2))-\varphi(2+2\delta).
$$
Denote the right-hand side by $c(\delta)$. As $\delta\to 0$, we have $\varphi(2+2\delta)=\varphi(2)+2\delta \varphi'(2)+o(\delta)$. Hence, as $\delta\to 0$,
\begin{align*}
c(\delta)
=c+\delta (c+\varphi(2)-2\varphi'(2))+o(\delta)
=c+\delta(c-c_2)+o(\delta).
\end{align*}
Since $c>c_2$ by the assumption of Theorem~\ref{theo:normal}, we can choose $\delta>0$ such that $c(\delta)>c$. It follows that there is $\eps>0$ such that  $c_n(\delta)>c+\eps$ provided that $n$ is sufficiently large. Recalling~\eqref{eq:est1}, we obtain for large $n$,
$$
N_n\E[W_n^{2+2\delta}] < N_n e^{-(c+\eps)n}.
$$
By~\eqref{eq:asympt_N}, the right-hand side tends to $0$ as $n\to\infty$. This proves~\eqref{eq:lim_n_E_Wn_delta} and~\eqref{eq:clt_cond2a}.

\subsection{Proof of~\eqref{eq:clt_cond3}}
Again using Lemma~4.2 of~\cite{ben_arous_etal05}, we obtain that for every $\delta>0$,
$$
N_n \E[W_n 1_{W_n>\tau}]\leq N_n \tau^{-1-2\delta} \E[W_n^{2+2\delta}].
$$
By~\eqref{eq:lim_n_E_Wn_delta}, the right-hand side converges to $0$ as $n\to\infty$.

\section{Proof of Theorem~\ref{theo:crit}}\label{sec:proof_crit}
The proof follows the same idea as the proof of Theorem~\ref{theo:normal}, but requires more subtle estimates. Let $W_n$ and $b_n$ be as in~\eqref{eq:def_bn}.
By Theorem~4.3 on p.~119 of~\cite{petrov_book95}, it suffices to check that for every $\tau>0$, the following three conditions hold:
\begin{align}
&\lim_{n\to\infty} N_n\P[W_n>\tau]=0, \label{eq:crit_cond1}\\
&\lim_{n\to\infty} N_n \Var[W_n 1_{W_n\leq \tau}]=1/2, \label{eq:crit_cond2}\\
&\lim_{n\to\infty} N_n \E[W_n 1_{W_n> \tau}]=0. \label{eq:crit_cond3}
\end{align}

\subsection{Proof of~\eqref{eq:crit_cond1}}
Set $\beta_n=(b_n+\log \tau)/n$ and $\beta=\varphi'(2)$. Note that by~\eqref{eq:def_c1_c2}, $I(\beta)=2\beta-\varphi(2)$ and, by Lemma~\ref{lem:I_prime}, $I'(\beta)=2$. By convexity of $I$ and~\eqref{eq:def_bn}, we have
\begin{equation}\label{eq:I_beta_n_est1}
I(\beta_n)\geq I(\beta)+I'(\beta)(\beta_n-\beta)=\frac 1n (\log N_n+2\log \tau).
\end{equation}
Note that by~\eqref{eq:def_bn} and~\eqref{eq:def_c1_c2},
\begin{equation}\label{eq:lim_beta_n_eq_beta}
\lim_{n\to\infty}\beta_n= \frac 12 (c_2+\varphi(2))=\varphi'(2)=\beta.
\end{equation}
Note also that $\beta=\varphi'(2)\in (\beta_0,\beta_{\infty})$. Hence, Lemma~\ref{cor:ld} yields
\begin{align*}
\P[W_n>\tau]
&=\P[S_n>n \beta_n]\\
&\leq C n^{-1/2} e^{-I(\beta_n) n } \\
&\leq C n^{-1/2} e^{-\log N_n-2\log \tau}\\
&=o(N_n^{-1}), \;\;\; n\to\infty.
\end{align*}

\subsection{Proof of~\eqref{eq:crit_cond2}}\label{sec:proof_crit2}
Exactly as in the proof of Theorem~\ref{theo:normal}, see~\eqref{eq:clt_cond2b}, it can be shown that for every $\tau>0$,
$$
\lim_{n\to\infty} N_n^{1/2} \E[W_n 1_{W_n\leq \tau}]=0.
$$
In view of $\E[W_n^2]=1/N_n$, see~\eqref{eq:ewn2}, in order to prove~\eqref{eq:crit_cond2}, it suffices to show that for every $\tau>0$,
\begin{equation}\label{eq:crit_cond2b}
\lim_{n\to\infty} N_n \E[W_n^2 1_{W_n\leq \tau}]=1/2.
\end{equation}
Recalling that $b_n$ is given by~\eqref{eq:def_bn}, we write
\begin{align*}
\lim_{n\to\infty} N_n \E[W_n^2 1_{W_n\leq \tau}]
&=\lim_{n\to\infty} e^{\log N_n-2b_n}\E[e^{2S_n}1_{S_n\leq b_n+\log \tau}]\\
&=\lim_{n\to\infty} e^{-\varphi(2)n} \E[e^{2S_n}1_{S_n\leq b_n+\log \tau}].
\end{align*}
Applying Part~\ref{p:1_lem3} of Lemma~\ref{lem:trunc_clt} below to the right-hand side yields~\eqref{eq:crit_cond2b}. The fact that the sequence $b_n+\log\tau$ satisfies the assumption of Part~\ref{p:1_lem3} of Lemma~\ref{lem:trunc_clt} can be seen as follows:
By~\eqref{eq:asympt_N}, we have, as $n\to\infty$,
$$
b_n=\frac{\log N_n+\varphi(2)n}{2}=\frac{c_2n+\varphi(2)n}{2}+o(1)=\varphi'(2)n+o(1).
$$
Hence, $b_n+\log\tau=\varphi'(2)n+O(1)$ as $n\to\infty$.

\subsection{Proof of~\eqref{eq:crit_cond3}}
Denote by $F_n$  the distribution function of the random variable $S_n-b_n$, and let $\bar F_n(y)=1-F_n(y)$ be the tail of $F_n$. To shorten the notation, we write $t=\log \tau$. We have
$$
\E[W_n 1_{W_n>\tau}]
=\E[e^{S_n-b_n}1_{S_n-b_n>t}]
=\int_{t}^{\infty}e^{y}dF_n(y).
$$
We take a sufficiently small but fixed $\eps>0$ and write
\begin{equation}\label{eq:dec_j1_j2}
N_n\E[W_n 1_{W_n>\tau}]=J_n^{(1)}(t,\eps)+J_n^{(2)}(\eps),
\end{equation}
where
$$
J_n^{(1)}(t,\eps)=N_n\int_{t}^{\eps n} e^{y}dF_n(y),\;\;\; J_n^{(2)}(\eps)=N_n\int_{\eps n}^{\infty} e^{y}dF_n(y).
$$
First we prove that for fixed $t\in\R$ and $\eps>0$,
\begin{equation}\label{eq:lim_j1}
\lim_{n\to\infty} J_n^{(1)}(t,\eps)=0.
\end{equation}
Using that $d\bar F_n(y)=-dF_n(y)$ and integrating by parts, we obtain
\begin{equation}\label{eq:int_parts1}
J_n^{(1)}(t,\eps)
=-N_n\int_{t}^{\eps n}e^{y}d\bar F_n(y)
=-N_n \bar F_n(y)e^y\Bigr|_{t}^{\eps n}+ \int_{t}^{\eps n} N_n\bar F_n(y)e^{y}dy.
\end{equation}
Take any $y\in[t,\eps n]$ and let $\beta_n=(b_n+y)/n$. As in~\eqref{eq:lim_beta_n_eq_beta}, $\lim_{n\to\infty}\beta_n=\beta$. It follows that for sufficiently large $n$, $\beta_n\in [\beta-2\eps,\beta+2\eps]$. If $\eps$ is sufficiently small, then Lemma~\ref{cor:ld} implies that for some constant $C>0$,
\begin{align*}
N_n\bar F_n(y)
=N_n\P[S_n>n \beta_n]
\leq C N_n  n^{-1/2} e^{-I(\beta_n) n }.
\end{align*}
As in~\eqref{eq:I_beta_n_est1},
\begin{equation}\label{eq:I_beta_n_est2}
I(\beta_n)\geq  \frac 1n (\log N_n+2y).
\end{equation}
It follows that for all $y\in [t,\eps n]$,
$$
N_n\bar F_n(y) \leq  C n^{-1/2} e^{-2 y}.
$$
Inserting this into the right-hand side of~\eqref{eq:int_parts1} and letting $n\to\infty$ yields~\eqref{eq:lim_j1}.

To complete the proof of~\eqref{eq:crit_cond3}, it suffices to show that for every $\eps>0$,
\begin{equation}\label{eq:lim_j2}
\lim_{n\to\infty} J_n^{(2)}(\eps)=0.
\end{equation}
Using first Lemma~\ref{lem:ld_est0} and then~\eqref{eq:I_beta_n_est2}, we obtain  the following estimate:
\begin{equation}\label{eq:tail_est1}
N_n \bar F_n(y)\leq N_n e^{-nI(\beta_n)}\leq e^{-2y}.
\end{equation}
Let $K_n$ be the set $[\eps n,\infty)\cap \Z$. We have
\begin{align*}
J_n^{(2)}(\eps)
&=N_n\int_{\eps n}^{\infty} e^{y}dF_n(y)\\
&\leq  N_n \sum_{j\in K_n} e^{j}  (F_n(j)-F_n(j-1))\\
&\leq N_n \sum_{j\in K_n} e^{j} \bar F_n(j-1).
\end{align*}
By~\eqref{eq:tail_est1}, $N_n\bar F_n(j-1)\leq e^2e^{-2j}$. Hence,
$$
J_n^{(2)}(\eps)\leq e^2 \sum_{j\in K_n} e^{-j}.
$$
By letting $n\to\infty$, we obtain~\eqref{eq:lim_j2}.
The validity of~\eqref{eq:crit_cond3} follows from the decomposition~\eqref{eq:dec_j1_j2} together with~\eqref{eq:lim_j1} and~\eqref{eq:lim_j2}.

\subsection{Computations with truncated exponential moments--1}\label{sec:trunc_clt}
In this section we prove a lemma on the asymptotic behavior of truncated exponential moments of sums of i.i.d.\ random variables. Part~\ref{p:1_lem3} of the lemma was used in Section~\ref{sec:proof_crit2}, whereas Part~\ref{p:2_lem3} will be needed only in the proof of Remark~\ref{rem:trunc}. Recall that $\{X_i,i\in\N\}$ are i.i.d.\ copies of a random variable $X$ satisfying the moment condition~\eqref{eq:def_varphi}, and that $S_n=X_1+\ldots+X_n$ are their partial sums.
\begin{lemma}\label{lem:trunc_clt}
Let $\alpha>0$ and set $\beta=\varphi'(\alpha)$. Let $b_n$ be some sequence, and set
$$
M_{\alpha}(n)=e^{-\varphi(\alpha) n} \E[e^{\alpha S_n} 1_{S_n\leq b_n}].
$$
\begin{enumerate}
\item \label{p:1_lem3} Suppose that $b_n=\beta n+r_n$, where $r_n=o(\sqrt{n})$ as $n\to\infty$. Then
$$
\lim_{n\to\infty} M_{\alpha}(n)=\frac 12.
$$
\item \label{p:2_lem3} Suppose that $b_n=\beta n+r_n$, where $r_n=o(n^{1/4})$ as $n\to\infty$, and assume that the distribution of $X$ is non-lattice. Then, as $n\to\infty$,
$$
M_{\alpha}(n) = \frac 12 +\frac{1}{\sqrt{2\pi \varphi''(\alpha) n} }\left( r_n+\frac {\varphi'''(\alpha)}{6\varphi''(\alpha)} \right) +o\left(\frac 1{\sqrt n}\right).
$$
\end{enumerate}
\end{lemma}
\begin{proof}
The idea of the proof will be to use an exponential change of measure argument. This will allow us to reduce Part~\ref{p:1_lem3} of the lemma to a statement following from the central limit theorem. Part~\ref{p:2_lem3} will require the more subtle Theorem~\ref{theo:asympt_clt}.

Denote by $F$ the distribution function of $X$. Let $\{\tilde X_i, i\in\N\}$ be independent copies of a random variable $\tilde X$ whose distribution function $\tilde F$ is given by
\begin{equation}\label{eq:tilde_F_F}
\frac{d\tilde F(t)}{dF(t)}=e^{\alpha t-\varphi(\alpha)}.
\end{equation}
Denote by $F_n$ and $\tilde F_n$ the distribution functions of $S_n$ and $\tilde S_n:=\tilde X_1+\ldots+\tilde X_n$, respectively. Then it is well known (see e.g.\ the lemma on p.~74 of~\cite{durrett_book91}) that
$$
\frac{d\tilde F_n(t)}{dF_n(t)}=e^{\alpha t-\varphi(\alpha)n}.
$$
It follows that
\begin{equation}\label{eq:exp_change}
M_{\alpha}(n)
=\int_{-\infty}^{b_n}e^{\alpha t-\varphi(\alpha)n}dF_n(t)
=\int_{-\infty}^{b_n} d\tilde F_n(t) 
=\P[\tilde S_n\leq b_n]. 
\end{equation}

This reduces the study of the truncated exponential moment $M_{\alpha}(n)$ to the study of the probability $\P[\tilde S_n\leq b_n]$. Asymptotic properties of this  probability will be derived by means of the central limit theorem applied to the sum $\tilde S_n$.

First of all let us compute $\E \tilde X$ and $\Var \tilde X$. Recall that $\varphi(t)=\log \E[e^{tX}]$ and set $\tilde \varphi(t)=\log \E[e^{t\tilde X}]$.  By~\eqref{eq:tilde_F_F}, we have
$$
\tilde\varphi(t)=\log \int_{\R} e^{tx}d\tilde F(x)=\log \int_{\R} e^{tx}e^{\alpha x-\varphi(\alpha)}d F(x)=\varphi(t+\alpha)-\varphi(\alpha).
$$
Hence,
\begin{equation}\label{eq:E_tilde X}
\E \tilde X=\tilde\varphi'(0)=\varphi'(\alpha)=\beta,\;\;\;
\Var \tilde X = \tilde \varphi''(0)= \varphi''(\alpha).
\end{equation}
It follows that  $\E \tilde S_n =\beta n$ and $\Var \tilde S_n=\varphi''(\alpha) n$.
We may write
\begin{equation}\label{eq:S_n_Phi_n}
\P[\tilde S_n\leq b_n]=\P\left[\frac{\tilde S_n-\beta n}{\sqrt{\varphi''(\alpha) n}}\leq \frac{b_n-\beta n}{\sqrt{\varphi''(\alpha) n}}\right]=
\Phi_n(x_n),
\end{equation}
where we have used the notation
\begin{equation}\label{eq:def_Phi_n}
\Phi_n(x)=\P\left[\frac{\tilde S_n-\beta n}{\sqrt{\varphi''(\alpha) n}}\leq x\right],\;\;\; x_n=\frac {r_n}{\sqrt{\varphi''(\alpha)n}}.
\end{equation}

Now we are ready to prove Part~\ref{p:1_lem3} of the lemma. If $r_n=o(\sqrt{n})$, then $\lim_{n\to\infty}x_n=0$, and the central limit theorem applied to $\tilde S_n$ shows that $\Phi_n(x_n)$ converges to $1/2$ as $n\to\infty$. Hence,
$$
\lim_{n\to\infty}M_{\alpha}(n)=\lim_{n\to\infty}\P[\tilde S_n\leq b_n]=\lim_{n\to\infty}\Phi_n(x_n)=1/2.
$$

To prove Part~\ref{p:2_lem3} of the lemma, we need to obtain an asymptotic formula for the probability  $\P[\tilde S_n\leq b_n]$ with an error term of the form $o(1/\sqrt n)$. To this end, we apply Theorem~\ref{theo:asympt_clt} to the centered variables $\{\tilde X_i-\beta, i\in\N\}$.
Recalling~\eqref{eq:E_tilde X} and using the well-known equality of the third cumulant and the third centered moment (see e.g. Eq.~26 on p.~66 of~\cite{gnedenko_book54}), we obtain
\begin{equation}\label{eq:mu2_mu3}
\mu_2:=\E[(\tilde X-\beta)^2]=\varphi''(\alpha),\;\;\;
\mu_3:=\E[(\tilde X-\beta)^3]=\tilde \varphi'''(0)= \varphi'''(\alpha).
\end{equation}
Applying Theorem~\ref{theo:asympt_clt},
we obtain
$$
\Phi_n (x_n)
=
\Phi\left(\frac {r_n}{\sqrt{\varphi''(\alpha)n}} \right)+ \frac{\mu_3 }{6 \sqrt{2\pi n}\mu_2^{3/2}}+o\left(\frac 1{\sqrt n}\right), \;\;\; n\to\infty.
$$
The standard normal distribution function $\Phi(x)=(2\pi)^{-1/2}\int_{-\infty}^x e^{-t^2/2}dt$ satisfies
\begin{equation}\label{eq:Phi_asympt}
\Phi(x)=\frac 12 + \frac{x}{\sqrt{2\pi}}+O(x^2), \;\;\; x\to 0.
\end{equation}
Since $r_n=o(n^{1/4})$ by the assumption of Part~\ref{p:2_lem3}, we have $x_n^2=o(1/\sqrt n)$ as $n\to\infty$. Hence,
$$
\Phi_n(x_n)
=
\frac 12 +  \frac {r_n}{\sqrt{2\pi \varphi''(\alpha)n}} + \frac{\mu_3 }{6 \sqrt{2\pi n}\mu_2^{3/2}}+o\left(\frac 1{\sqrt n}\right), \;\;\; n\to\infty.
$$
To complete the proof of Part~\ref{p:2_lem3} of the lemma, recall that $\mu_2$ and $\mu_3$ are given by~\eqref{eq:mu2_mu3}.
\end{proof}

\begin{proof}[Proof of Remark~\ref{rem:trunc}]
By~\eqref{eq:def_An}, we have $A_n=N_n\E[e^{S_n}1_{S_n\leq b_n}]$, where $b_n$ is given by~\eqref{eq:def_Bn_large}. The statement of the remark follows by applying Part~\ref{p:2_lem3} of Lemma~\ref{lem:trunc_clt} with $\alpha=1$.
\end{proof}

\section{Proof of Theorem~\ref{theo:stable}}\label{sec:proof_stable}
By the statement of Theorem~\ref{theo:stable}, we assume that $c\in (0,c_2)$. Recall that $\alpha\in (0,2)$ is defined as  the solution of the equation $I(\varphi'(\alpha))=c$, that $\beta=\varphi'(\alpha)$, and that
\begin{equation}\label{eq:def_bn_stable}
b_n=\log B_n = \beta n -\alpha^{-1}\log \left( \alpha \sqrt{2\pi\varphi''(\alpha) n} \right).
\end{equation}
Define a positive-valued random variable $W_n$ by
\begin{equation}\label{eq:def_wn_1}
W_n=B_n^{-1}e^{S_n}=e^{S_n-b_n}.
\end{equation}
Let $W_{n,1},\ldots, W_{n,N_n}$ be independent copies of $W_n$. With this notation, Theorem~\ref{theo:stable} is equivalent to the following statement:
\begin{equation}\label{eq:stable_cond}
\sum_{i=1}^{N_n} W_{n,i}- B_n^{-1}A_n\todistr \FF_{\alpha},\;\;\; n\to\infty.
\end{equation}
By~\cite[Section~6]{ben_arous_etal05}, the convergence in~\eqref{eq:stable_cond} will be established once we have verified the validity of the following three statements:
\begin{enumerate}
\item For every $\tau>0$,
\begin{equation}\label{eq:stable_cond1}
\lim_{n\to\infty} N_n \P[W_n> \tau]=\tau^{-\alpha}.
\end{equation}
\item We have
\begin{equation}\label{eq:stable_cond2}
\lim_{\tau\to 0+}\limsup_{n\to\infty} N_n \Var [W_n 1_{W_n\leq \tau}]=0.
\end{equation}
\item  For every $\tau>0$,
\begin{equation}\label{eq:stable_cond3}
D_{\alpha}(\tau)
:=\lim_{n\to\infty} (N_n\E[W_n 1_{W_n\leq \tau}]-B_n^{-1}A_n)
=
\begin{cases}
\frac{\alpha}{1-\alpha}\tau^{1-\alpha}, & \text{ if }\alpha\neq 1,\\
\log \tau,\;\;\; &\text{ if } \alpha=1.
\end{cases}
\end{equation}
\end{enumerate}

Our proof of the above conditions will be based on lemmas whose statements and proofs are postponed to  Section~\ref{sec:comp_trunc_ld} below.

\subsection{Proof of~\eqref{eq:stable_cond1}}
Since $W_n=e^{S_n-b_n}$ by~\eqref{eq:def_wn_1}, we have
$$
N_n\P[W_n>\tau]=N_n\P[S_n-b_n>\log \tau].
$$
By Lemma~\ref{lem:ld} below with $y=\log \tau$, the right-hand side converges to $\tau^{-\alpha}$ as $n\to\infty$.

\subsection{Proof of~\eqref{eq:stable_cond2}} \label{sec:stable_cond1}
Since $\Var [W_n 1_{W_n\leq \tau}]\leq \E[W_n^2 1_{W_n\leq \tau}]$, it suffices to prove that
\begin{equation}\label{eq:stable_cond2a}
\lim_{\tau \to 0+} \limsup_{n\to\infty} N_n\E[W_n^2 1_{W_n\leq \tau}]=0.
\end{equation}
Recall that by~\eqref{eq:def_wn_1}, $W_n=e^{S_n-b_n}$. We have
\begin{align*}
 N_n\E[W_n^2 1_{W_n\leq \tau}]
= N_n\E[e^{2(S_n-b_n)}1_{S_n\leq b_n+\log \tau}].
\end{align*}
recall that $\alpha\in(0,2)$. Applying to the right-hand side Part~\ref{p:1_lem6} of Lemma~\ref{lem:aux_int} below with $\kappa=2$ and $t=\log \tau$, we obtain
$$
\lim_{n\to\infty} N_n\E[W_n^2 1_{W_n\leq \tau}]=\frac{\alpha}{2-\alpha} \tau^{2-\alpha}.
$$
Then~\eqref{eq:stable_cond2a} follows by letting $\tau\to 0+$.

\subsection{Proof of~\eqref{eq:stable_cond3}}
Let us assume first that $\alpha\in (0,1)$ (which is equivalent to $c\in (0,c_1)$). Then by~\eqref{eq:def_An}, $A_n=0$, and
$$
D_{\alpha}(\tau)= \lim_{n\to\infty} N_n\E[e^{S_n-b_n} 1_{S_n\leq b_n+\log\tau}].
$$
Applying to the right-hand side Part~\ref{p:1_lem6} of Lemma~\ref{lem:aux_int} below with $\kappa=1$ and $t=\log\tau$, we get the statement of~\eqref{eq:stable_cond3} in the case $\alpha\in(0,1)$.

Now assume that $\alpha\in(1,2)$ (equivalently, $c\in (c_1,c_2)$). By~\eqref{eq:def_An}, $B_n^{-1} A_n=N_n \E [W_n]$ and hence,
\begin{align*}
D_{\alpha}(\tau)
&=\lim_{n\to\infty} (N_n\E[W_n 1_{W_n\leq \tau}]-N_n \E [W_n])\\
&=-\lim_{n\to\infty} N_n\E[W_n 1_{W_n> \tau}]\\
&=-\lim_{n\to\infty} N_n\E[e^{S_n-b_n} 1_{S_n> b_n+ \log \tau}].
\end{align*}
Applying to the right-hand side Part~\ref{p:2_lem6} of Lemma~\ref{lem:aux_int} below with $\kappa=1$ and $t=\log\tau$, we get the statement of~\eqref{eq:stable_cond3} in the case $\alpha\in (1,2)$.

Finally, let us consider the case $\alpha=1$. By~\eqref{eq:def_An}, $B_n^{-1}A_n=N_n \E[W_n 1_{W_n\leq 1}]$. For concreteness, assume that $\tau\geq 1$, the proof in the case $\tau<1$ being analogous. Then
\begin{align*}
D_{\alpha}(\tau)
&=\lim_{n\to\infty} (N_n\E[W_n 1_{W_n\leq \tau}]-N_n \E[W_n 1_{W_n\leq 1}])\\
&=\lim_{n\to\infty} N_n\E[W_n 1_{W_n\in (1,\tau]}]\\
&=\lim_{n\to\infty} N_n\E[e^{S_n-b_n} 1_{b_n< S_n \leq b_n+ \log \tau}].
\end{align*}
Applying to the right-hand side Part~\ref{p:3_lem6} of Lemma~\ref{lem:aux_int} below with $\kappa=1$, $t_1=0$,  $t_2=\log\tau$, we get the statement of~\eqref{eq:stable_cond3} in the case $\alpha=1$.
This completes the proof of~\eqref{eq:stable_cond3}.

\subsection{Computations with truncated exponential moments--2} \label{sec:comp_trunc_ld}
In this section we prove several lemmas on truncated exponential moments for sums of i.i.d.\ variables. These lemmas were used in the above proof of Theorem~\ref{theo:stable}. Let $c_{\infty}=\lim_{\alpha\to+\infty}I(\varphi'(\alpha))$.
\begin{lemma}\label{lem:ld}
Suppose that~\eqref{eq:asympt_N} and~\eqref{eq:def_varphi} hold and that $c\in (0, c_{\infty})$. Let $\alpha$ be the solution of the equation $I(\varphi'(\alpha))=c$, set $\beta=\varphi'(\alpha)$, and let $b_n$ be defined by~\eqref{eq:def_bn_stable}. Assume also that the distribution of $X$ is non-lattice.
Then for every $y\in\R$ we have
\begin{equation}\label{eq:lem_aux_1}
\lim_{n\to\infty} N_n\P[S_n-b_n>y]=e^{-\alpha y}.
\end{equation}
As long as $y$ stays bounded, the convergence in~\eqref{eq:lem_aux_1} is uniform in $y$.
\end{lemma}
\begin{proof}
Let $\beta_n=(b_n+y)/n$ and note that $\lim_{n\to\infty}\beta_n=\beta$.
By Theorem~\ref{theo:ld} (see in particular, the uniformity statement in its formulation), we have
\begin{equation}\label{eq:P_Sn_bn}
\P[S_n-b_n > y]
\sim
\frac{1}{\alpha \sqrt{2\pi \varphi''(\alpha) n }} e^{-n I(\beta_n)}, \;\;\; n\to\infty.
\end{equation}
To complete the proof, we need an asymptotic formula for $nI(\beta_n)$ as $n\to\infty$ with an error term of order $o(1)$. Using Taylor's expansion of the function $I$ at the point $\beta$ we obtain
$$
I(\beta_n)=I(\beta)+I'(\beta)(\beta_n-\beta) + O((\beta_n-\beta)^2),\;\;\; n\to\infty.
$$
Recall that $I(\beta)=c$ and by Lemma~\ref{lem:I_prime}, $I'(\beta)=\alpha$.  Noting that by~\eqref{eq:def_bn_stable}, $\beta_n-\beta=o(1/\sqrt n)$ as $n\to\infty$, we obtain
\begin{align}\label{eq:i_beta_n_0}
nI(\beta_n)=cn - \log\left(\alpha \sqrt{2\pi\varphi''(\alpha)n}\right) + \alpha y +o(1), \;\;\; n\to\infty.
\end{align}
Applying~\eqref{eq:i_beta_n_0} to the right-hand side of~\eqref{eq:P_Sn_bn}, we obtain that
$$
\P[S_n-b_n > y]  \sim e^{-cn}e^{-\alpha y}, \;\;\; n\to\infty.
$$
To complete the proof of~\eqref{eq:lem_aux_1}, recall that by~\eqref{eq:asympt_N}, $N_n\sim e^{cn}$ as $n\to\infty$.  \end{proof}
The next lemma gives an estimate for the probability $\P[S_n-b_n>y]$ which is valid uniformly for $y\in [-\eps n, \eps n]$, where $\eps>0$ is a fixed small number.
\begin{lemma}\label{lem:ld_est1}
Suppose that~\eqref{eq:asympt_N} and~\eqref{eq:def_varphi} hold and that $c\in (0,c_{\infty})$.  Let $\alpha$ be the solution of the equation $I(\varphi'(\alpha))=c$, set $\beta=\varphi'(\alpha)$, and let $b_n$ be defined by~\eqref{eq:def_bn_stable}.
If $\eps>0$ is sufficiently small, then there is a constant $C>0$ such that for all $n\in\N$ and all $y\in [-\eps n, \eps n]$,
\begin{equation}\label{eq:ld1}
N_n \P[S_n-b_n>y]< C  e^{-\alpha y}.
\end{equation}
\end{lemma}
\begin{proof}
We have $\P[S_n-b_n>y]=\P[S_n>n\beta_n]$, where $\beta_n=(b_n+y)/n$. Since it is assumed that $y\in [-\eps n, \eps n]$ and since by~\eqref{eq:def_bn_stable}, $\lim_{n\to\infty} b_n/n=\beta$, we have $\beta_n\in (\beta-2\eps, \beta+2\eps)$ provided that $n$ is sufficiently large. By convexity of the function $I$,
\begin{equation}\label{eq:n_I_beta_n_ineq}
nI(\beta_n)
\geq n I(\beta)+ nI'(\beta)(\beta_n-\beta)
=cn- \log \left( \alpha \sqrt{2\pi\varphi''(\alpha) n} \right)+\alpha y.
\end{equation}
Using first Lemma~\ref{cor:ld} and then~\eqref{eq:n_I_beta_n_ineq}, we obtain
$$
\P[S_n>n\beta_n] \leq  \frac{C_1}{\sqrt n} e^{-nI (\beta_n)}\leq  \frac{C_1}{\sqrt n} e^{- cn }e^{\log ( \alpha \sqrt{2\pi\varphi''(\alpha) n})} e^{-\alpha y}\leq C_2 N_n^{-1} e^{-\alpha y}.
$$
This completes the proof of~\eqref{eq:ld1}.
\end{proof}

\begin{lemma}\label{lem:aux_int}
Let the assumptions of Lemma~\ref{lem:ld} be satisfied.  Then the following three statements hold true.
\begin{enumerate}
\item \label{p:1_lem6} Let $\kappa>\alpha$. Then for every $t\in\R$,
\begin{equation}\label{eq:lem_aux_2}
\lim_{n\to\infty} N_n \E[e^{\kappa (S_n-b_n)}1_{S_n\leq b_n+t}]=\frac{\alpha}{\kappa-\alpha} e^{(\kappa-\alpha)t}.
\end{equation}
\item \label{p:2_lem6} Let $0<\kappa<\alpha$. Then for every $t\in\R$,
\begin{equation}\label{eq:lem_aux_3}
\lim_{n\to\infty} N_n \E[e^{\kappa (S_n-b_n)}1_{S_n>b_n+t}]=-\frac{\alpha}{\kappa-\alpha} e^{(\kappa-\alpha)t}.
\end{equation}
\item \label{p:3_lem6} Let $\kappa=\alpha$. Then for every $t_1\leq t_2$,
\begin{equation}\label{eq:lem_aux_4}
\lim_{n\to\infty} N_n \E[e^{\kappa (S_n-b_n)}1_{b_n+t_1<S_n\leq b_n+t_2}]=\kappa (t_2-t_1).
\end{equation}
\end{enumerate}
\end{lemma}
\begin{proof}
Let us start by  giving a non-rigorous proof of Part~\ref{p:1_lem6}.  To shorten the notation, we set
$$
J_n(t)=N_n \E[e^{\kappa (S_n-b_n)}1_{S_n\leq b_n+t}].
$$
Let $F_n(y)=\P[S_n-b_n\leq y]$ be the distribution function of the random variable $S_n-b_n$. By Lemma~\ref{lem:ld}, we have for each fixed $y\in\R$,
$$
1-F_n(y) \sim N_n^{-1} e^{-\alpha y},\;\;\; n\to\infty.
$$
Differentiating this formally, we obtain $dF_n(y)\sim  N_n^{-1} \alpha e^{-\alpha y}dy$ as $n\to\infty$. Now we can compute the limit in~\eqref{eq:lem_aux_2} as follows:
$$
J_n(t)
= N_n\int_{-\infty}^t e^{\kappa y}dF_n(y)
\sim \int_{-\infty}^t \alpha e^{(\kappa-\alpha) y}  dy
=\frac{\alpha}{\kappa-\alpha} e^{(\kappa-\alpha)t}, \;\;\;n\to\infty.
$$
Note that the integral $\int_{-\infty}^t e^{(\kappa-\alpha) y}  dy$ is finite since $\kappa>\alpha$. In a similar way, it is also possible to give non-rigorous proofs of Parts~\ref{p:2_lem6} and~\ref{p:3_lem6} of the lemma. However, making these arguments precise requires some work.

We start by proving~\eqref{eq:lem_aux_2}. Take a large number $T>0$ and a small number $\eps>0$, and set $\Delta=\beta-\beta_0-\eps$ (recall that $\beta_0$ was defined in~\eqref{eq:def_beta_0}). We have a decomposition
\begin{equation} \label{eq:dec_In}
J_n(t)
=N_n \int_{-\infty}^t e^{\kappa y}dF_n(y)
=J_n^{(1)}(t,T)+J_n^{(2)}(T,\eps)+J_n^{(3)}(\eps)+J_n^{(4)}(\eps),
\end{equation}
where
\begin{align}
J_n^{(1)}(t,T)&=\int_{-T}^{t}N_n e^{\kappa y}dF_n(y),\\
J_n^{(2)}(T,\eps)&=\int_{-\eps n}^{-T} N_n e^{\kappa y}dF_n(y),\\
J_n^{(3)}(\eps)&=\int_{-\Delta n}^{-\eps n} N_ne^{\kappa y}dF_n(y),\\
J_n^{(4)}(\eps)&=\int_{-\infty}^{-\Delta n} N_ne^{\kappa y}dF_n(y).
\end{align}
First let us show that
\begin{equation}\label{eq:In1}
\lim_{T\to+\infty}\lim_{n\to\infty} J_n^{(1)}(t,T)=\frac{\alpha}{\kappa-\alpha}e^{(\kappa-\alpha)t}.
\end{equation}
Let $\bar F_n(y)=1-F_n(y)$ be the tail of $F_n$. Noting that $d\bar F_n(y)=-dF_n(y)$ and integrating by parts, we obtain
\begin{align}
J_n^{(1)}(t,T)
&=-\int_{-T}^{t} e^{\kappa y} N_n d\bar F_n(y) \label{eq:part_int_1}\\
&=-e^{\kappa y}N_n \bar F_n(y)\Bigr|_{-T}^t+\kappa \int_{-T}^t N_n\bar F_n(y)e^{\kappa y}dy. \notag
\end{align}
By Lemma~\ref{lem:ld}, the following relation holds uniformly in $y\in[-T,t]$:
\begin{equation}\label{eq:lim_N_bar_Fn}
\lim_{n\to\infty} N_n\bar F_n(y)=e^{-\alpha y}.
\end{equation}
Applying~\eqref{eq:lim_N_bar_Fn} to the right-hand side of~\eqref{eq:part_int_1}, we obtain
$$
\lim_{n\to\infty} J_n^{(1)}(t,T)
= -e^{(\kappa-\alpha)y}\Bigr|_{-T}^t+ \kappa \int_{-T}^t e^{(\kappa-\alpha)y}dy
=\frac{\alpha}{\kappa-\alpha} e^{(\kappa-\alpha)y}\Bigr|_{-T}^t.
$$
After letting $T\to +\infty$ this yields~\eqref{eq:In1}.

To complete the proof of~\eqref{eq:lem_aux_2}, we need to show that the terms $J_n^{(2)}(T,\eps)$, $J_n^{(3)}(\eps)$, and $J_n^{(4)}(\eps)$ are in some sense negligible.  First we prove that for each fixed $\eps>0$,
\begin{equation}\label{eq:In2}
\lim_{T\to+\infty}\limsup_{n\to\infty} J_n^{(2)}(T,\eps)=0.
\end{equation}
Integrating by parts as in~\eqref{eq:part_int_1}, we obtain
\begin{align}
J_n^{(2)}(T,\eps)
&=-e^{\kappa y} N_n \bar F_n(y)\Bigr|_{-\eps n}^{-T}+\kappa \int_{-\eps n}^{-T} N_n \bar F_n(y)e^{\kappa y}dy. \label{eq:j2_est}
\end{align}
By Lemma~\ref{lem:ld_est1}, there is a constant $C>0$ such that for all $n\in\N$ and all $y\in [-\eps n, 0]$,
$$
N_n\bar F_n(y)\leq C e^{-\alpha y}.
$$
Applying this to~\eqref{eq:j2_est} shows that for every $n\in\N$ the following  estimate holds:
$$
J_n^{(2)}(T,\eps)
\leq C \left(e^{-(\kappa-\alpha)T}+e^{-(\kappa-\alpha)\eps n}+\kappa \int_{-\eps n}^{-T} e^{(\kappa-\alpha)y}dy\right).
$$
Then~\eqref{eq:In2} follows by letting $n, T\to +\infty$ and recalling that $\kappa>\alpha$.

In the next step, we prove that for each fixed $\eps>0$,
\begin{equation}\label{eq:In3}
\lim_{n\to\infty} J_n^{(3)}(\eps)=0.
\end{equation}
We will prove that there exists a constant $C>0$ such that for all $n\in\N$ and all $y\in[-\Delta n-1, -\eps n]$,
\begin{equation}\label{eq:tail_est2}
\bar F_n(y)\leq C n^{1/2} N_n^{-1} e^{-\alpha y}.
\end{equation}
We may write $\bar F_n(y)=\P[S_n>n \beta_n]$, where $\beta_n=(b_n+y)/n$. Note that
\begin{equation}\label{eq:ineq_888}
\liminf_{n\to\infty}\beta_n\geq \beta-\Delta=\beta_0+\eps.
\end{equation}
Using the convexity of the function $I$ as in~\eqref{eq:n_I_beta_n_ineq}, we obtain
\begin{equation}\label{eq:n_I_beta_n_ineq1}
nI(\beta_n)
\geq cn+\alpha y - \log \left( \alpha \sqrt{2\pi\varphi''(\alpha) n}\right).
\end{equation}
Then~\eqref{eq:tail_est2} can be proved as follows: by Lemma~\ref{lem:ld_est0} (which is applicable in view of~\eqref{eq:ineq_888}) and~\eqref{eq:n_I_beta_n_ineq1},
$$
\bar F_n(y)\leq e^{-n I(\beta_n)}\leq  e^{-cn}e^{-\alpha y}e^{\log (\alpha \sqrt{2\pi\varphi''(\alpha) n})}< C n^{1/2} N_{n}^{-1} e^{-\alpha y}.
$$
Now we are in position to start the proof of~\eqref{eq:In3}. Denote by $K_n$ be the set $[-\Delta n-1,-\eps n]\cap\Z$. It follows that
\begin{align*}
J_n^{(3)}(\eps)
&=
N_n \int_{-\Delta n}^{-\eps n}e^{\kappa y}dF_n(y)\\
&\leq
N_n \sum_{j\in K_n} e^{\kappa (j+1)} (F_n(j+1)-F_n(j))\\
&\leq
N_n \sum_{j\in K_n} e^{\kappa (j+1)} \bar F_n(j)\\
&\leq
C n^{1/2} \sum_{j\in K_n} e^{(\kappa-\alpha) j }\\
&\leq
C' n^{1/2} e^{-\eps (\kappa-\alpha) n}.
\end{align*}
Since $\kappa-\alpha>0$, the right-hand side converges to $0$ as $n\to\infty$. This proves~\eqref{eq:In3}.

Finally, let us show that for sufficiently small $\eps>0$,
\begin{equation}\label{eq:In4}
\lim_{n\to\infty} J_n^{(4)}(\eps)=0.
\end{equation}
We have
\begin{equation}\label{eq:ineq_143}
J_n^{(4)}(\eps)=\int_{-\infty}^{-\Delta n} N_ne^{\kappa y}dF_n(y)\leq N_n e^{-\kappa \Delta n}.
\end{equation}
So, it suffices to show that $\kappa \Delta>c$.
This can be done as follows. Recall that $\beta_0=\lim_{t\to 0+}\varphi'(t)$. By convexity of $\varphi$, $\varphi(\alpha)\geq \alpha \beta_0$. Using this, we obtain
$$
c=\alpha\varphi'(\alpha)-\varphi(\alpha)\leq \alpha\varphi'(\alpha)-\alpha \beta_0=\alpha (\beta-\beta_0)<\kappa (\beta-\beta_0).
$$
Thus, if $\eps$ is sufficiently small, then $c<\kappa (\beta-\beta_0-\eps)=\kappa \Delta$. This completes the proof of~\eqref{eq:In4}.

Now, the proof of Part~\ref{p:1_lem6} of the lemma can be completed by letting $T\to +\infty$ in the decomposition~\eqref{eq:dec_In} and taking into account~\eqref{eq:In1}, \eqref{eq:In2}, \eqref{eq:In3}, \eqref{eq:In4}. The proof of Part~\ref{p:2_lem6} is similar and will be therefore omitted.

Let us prove Part~\ref{p:3_lem6}. Define
$$
J_n(t_1, t_2)=N_n \E[e^{\kappa (S_n-b_n)}1_{b_n+t_1<S_n\leq b_n+t_2}].
$$
Integration by parts yields
\begin{align*}
J_n(t_1,t_2)
&=N_n\int_{t_1}^{t_2} e^{\kappa y} dF_n(y)\\
&=-N_n\int_{t_1}^{t_2} e^{\kappa y} d\bar F_n(y)\\
&=-e^{\kappa y}N_n\bar F_n(y)\Bigr |_{t_1}^{t_2}+\kappa \int_{t_1}^{t_2} e^{\kappa y}N_n\bar F_n(y) dy.
\end{align*}
Recalling~\eqref{eq:lim_N_bar_Fn} and the assumption $\kappa=\alpha$, we obtain $\lim_{n\to\infty} N_n \bar F_n(y)=e^{-\kappa y}$, and the convergence is uniform in $y$ as long as $y$ stays bounded (see Lemma~\ref{lem:ld}). It follows that
$$
\lim_{n\to\infty} J_n(t_1,t_2)=\kappa (t_2-t_1),
$$
which proves Part~\ref{p:3_lem6} of the lemma.
\end{proof}

\section{Proof of Theorem~\ref{theo:stable_lattice}}\label{sec:proof_stable_lat}
Let $W_n=e^{S_n-b_n'-\Delta}$, where $b_n'=[b_n]_h$, and let $W_{n,1},\ldots, W_{n,N_n}$ be independent copies of $W_n$.
Then Theorem~\ref{theo:stable_lattice} is equivalent to the following statement:
\begin{equation}\label{eq:stable_latt_cond}
\sum_{i=1}^{N_{n_k}} W_{{n_k},i}- B_{n_k}^{-1}A_{n_k}\todistr \FF'_{\alpha, \Delta},\;\;\; k\to\infty.
\end{equation}
Note that the random variable $W_n$ takes values in the set $\exp(h\Z-\Delta)$. By the standard theory of convergence to infinitely divisible distributions, see e.g.~\cite[\S 25]{gnedenko_book54}, the convergence in~\eqref{eq:stable_cond} will be established once we have verified the validity of the following three statements:
\begin{enumerate}
\item For every $x\in \exp(h\Z-\Delta)$,
\begin{equation}\label{eq:stable_latt_cond1}
\lim_{k\to\infty} N_{n_k} \P[W_{n_k}=x]= x^{-\alpha}.
\end{equation}
\item For the truncated variance, we have
\begin{equation}\label{eq:stable_latt_cond2}
\lim_{\tau\to 0+}\limsup_{k\to\infty} N_{n_k} \Var [W_{n_k} 1_{W_{n_k}\leq \tau}]=0.
\end{equation}
\item  For every $\tau\notin \exp(h\Z-\Delta)$, the following limit exists and is finite:
\begin{equation}\label{eq:stable_latt_cond3}
D_{\alpha}(\tau):=\lim_{k\to\infty} (N_{n_k}\E[W_{n_k} 1_{W_{n_k}\leq \tau}]-B_{n_k}^{-1}A_{n_k}).
\end{equation}
\end{enumerate}

\subsection{Proof of~\eqref{eq:stable_latt_cond1}}
First note that $\P[W_n=x]=\P[S_n=n \beta_n]$, where $\beta_n=(b_n'+\Delta+\log x)/n$. Note also that by~\eqref{eq:def_bn_lattice}, $\lim_{n\to\infty}\beta_n=\beta$.
By Part~\ref{p:2_th_petrov} of Theorem~\ref{theo:ld},
\begin{equation}\label{eq:P_Sn_bn_latt}
\P[W_n=x] \sim \frac{h}{\sqrt {2\pi \varphi''(\alpha) n}} e^{-n I(\beta_n)}.
\end{equation}
To prove~\eqref{eq:stable_latt_cond1}, we need to find an asymptotic formula for $nI(\beta_n)$ as $n\to\infty$ with an error term of the form $o(1)$. We have, by Taylor's expansion,
$$
I(\beta_n)=I(\beta)+I'(\beta)(\beta_n-\beta) + O((\beta_n-\beta)^2),\;\;\; n\to\infty.
$$
Recall that $\beta=\varphi'(\alpha)$ and hence, $I(\beta)=c$. By Lemma~\ref{lem:I_prime}, $I'(\beta)=\alpha$. Further, recall that $b_n$ is defined by~\eqref{eq:def_bn_lattice} and hence, $\beta_n-\beta=o(n^{-1/2})$ as $n\to\infty$. Using all these facts, we obtain, as $n\to\infty$,
\begin{align*}
nI(\beta_n)
&=cn + \alpha (b_n'+\Delta-n\beta) + \alpha \log x +o(1)\notag \\
&=cn- \log (h^{-1}\sqrt{2\pi \varphi''(\alpha)n})+\alpha (\Delta-\Delta_n)+\alpha \log x +o(1).
\end{align*}
By the assumption of the theorem, $\lim_{k\to\infty}\Delta_{n_k}=\Delta$. Hence,
\begin{equation}\label{eq:i_beta_n_latt}
n_kI(\beta_{n_k})
=cn_k- \log (h^{-1}\sqrt{2\pi \varphi''(\alpha)n_k})+\alpha \log x +o(1),\;\;\; k\to\infty.
\end{equation}
Applying~\eqref{eq:i_beta_n_latt} to the right-hand side of~\eqref{eq:P_Sn_bn_latt}, we obtain that
$$
\P[W_{n_k}=x]  \sim e^{-cn_k} x^ {-\alpha}, \;\;\; k\to\infty.
$$
To complete the proof, recall that by~\eqref{eq:asympt_N}, $N_n\sim e^{cn}$ as $n\to\infty$.

\subsection{Proof of~\eqref{eq:stable_latt_cond2}}
It suffices to show that
\begin{equation}\label{eq:111}
\lim_{\tau \to +0} \lim_{k\to\infty} N_{n_k}\E[W_{n_k}^2 1_{W_{n_k}\leq \tau}]=0.
\end{equation}
Recall that $W_n$ takes values in the set $\exp(h\Z-\Delta)$. Using this and then~\eqref{eq:stable_latt_cond1}, we obtain
\begin{align}
\lim_{k\to\infty}N_{n_k}\E[W_{n_k}^2 1_{W_{n_k}\leq \tau}]
&=\lim_{k\to\infty}\sum_{\genfrac{}{}{0pt}{1} {x\in \exp(h\Z-\Delta)} {x\leq \tau} } x^2  N_{n_k} \P[W_{n_k}=x] \notag\\
&=\sum_{\genfrac{}{}{0pt}{1} {x\in \exp(h\Z-\Delta)} {x\leq \tau}}  x^{2-\alpha}. \label{eq:844}
\end{align}
We have omitted the justification of interchanging the limit and the sum, since it can be done as in the proof of Part~\ref{p:1_lem6} of Lemma~\ref{lem:aux_int}. Indeed, this proof is based on Lemmas~\ref{cor:ld} and~\ref{lem:ld_est0} which are valid both in the lattice and in the non-lattice case, and on Lemma~\ref{lem:ld} which, in the non-lattice case, should be replaced by~\eqref{eq:stable_latt_cond1}.
To complete the proof of~\eqref{eq:stable_latt_cond2}, let $\tau\to +0$ in~\eqref{eq:844} and recall that $2-\alpha>0$.


\subsection{Proof of~\eqref{eq:stable_latt_cond3}}
Assume that $\alpha\in(0,1)$. Then
\begin{align*}
N_{n_k} \E[W_{n_k}1_{W_{n_k}\leq \tau}]
&= \sum_{\genfrac{}{}{0pt}{1} {x\in \exp(h\Z-\Delta)} {x\leq \tau} } x  N_{n_k} \P[W_{n_k}=x].
\end{align*}
Using~\eqref{eq:stable_latt_cond1} and again omitting the justification of interchanging the sum and the limit, we obtain
$$
\lim_{k\to\infty} N_{n_k} \E[W_{n_k}1_{W_{n_k}\leq \tau}] = \sum_{\genfrac{}{}{0pt}{1} {x\in \exp(h\Z-\Delta)} {x\leq \tau}}  x^{1-\alpha}.
$$
Note that the right-hand side is finite since $\alpha\in(0,1)$. This proves~\eqref{eq:stable_latt_cond3} in the case $\alpha\in (0,1)$. We omit the cases $\alpha\in (1,2)$ and $\alpha=1$, since they can be handled similarly.

\section{Proofs of the weak laws of large numbers} \label{sec:proof_lln}

\subsection{Proof of Theorem~\ref{theo:weak_lln}}
We will deduce Theorem~\ref{theo:weak_lln} from the results of Theorems~\ref{theo:normal}, \ref{theo:crit}, \ref{theo:stable}, \ref{theo:stable_lattice} on the limiting distributions of $Z_n$, and the following standard lemma.
\begin{lemma}\label{lem:weak_comp}
Let $Z_n$ be a sequence of random variables, and let $A_n,B_n\neq 0$ be two sequences of normalizing constants such that the following two conditions are satisfied:
\begin{enumerate}
\item The random variables $\{(Z_n-A_n)/B_n, n\in\N\}$ form a tight sequence.
\item We have $\lim_{n\to\infty} B_n/A_n=0$.
\end{enumerate}
Then $Z_n/A_n\toprobab 1$ as $n\to\infty$.
\end{lemma}
\begin{proof}
Fix some $\eps>0$ and $\delta>0$. We have to show that there is $N=N(\eps,\delta)$ such that for all $n>N$,
$$
\P\left[\left|\frac{Z_n}{A_n}-1\right|>\eps\right]<\delta.
$$
By the first assumption of the lemma, we can find $m=m(\delta)$ such that for all $n\in\N$,
$$
\P\left[\left|\frac{Z_n-A_n}{B_n}\right|>m\right]<\delta.
$$
By the second assumption of the lemma, there is $N=N(\eps,\delta)$ such that for $n>N$, we have $\eps |A_n|/|B_n|> m$. This implies that for $n>N$,
$$
\P\left[\left|\frac{Z_n}{A_n}-1\right|>\eps \right]=\P\left[\left|\frac{Z_n-A_n}{B_n}\right|>\eps \frac{|A_n|}{|B_n|}\right]\leq \P\left[\left|\frac{Z_n-A_n}{B_n}\right|>m\right]< \delta.
$$
This completes the proof of the lemma.
\end{proof}

\begin{proof}[Proof of Theorem~\ref{theo:weak_lln}]
First, we define two normalizing sequences $A_n$ and $B_n$ as follows. If $c\in[c_2,\infty)$, then we set $A_n=\E Z_n$ and  $B_n=(\Var Z_n)^{1/2}$. If $c\in [c_1,c_2)$, then let $A_n, B_n$ be defined as in Theorem~\ref{theo:stable} if $X$ is non-lattice and as in Theorem~\ref{theo:stable_lattice} if $X$ is lattice.
Then, as $n\to\infty$,
\begin{equation}\label{eq:an_cases}
A_n= \begin{cases} \E Z_n \sim e^{(\varphi(1)+c)n}, &\text{ if }c>c_1,\\
N_n\E[e^{S_n}1_{S_n\leq b_n}] \sim \frac 12 e^{(\varphi(1)+c_1)n}\sim \frac 12 e^{\varphi'(1)n},
&\text{ if }c=c_1.\end{cases}
\end{equation}
Note that the case $c=c_1$ follows from Part~\ref{p:1_lem3} of Lemma~\ref{lem:trunc_clt}. Using this, we may rewrite~\eqref{eq:weak_lln} as follows:
\begin{equation}\label{eq:plim_zn_an}
Z_n/A_n\toprobab 1,\;\;\; n\to\infty.
\end{equation}
By Theorems~\ref{theo:normal}, \ref{theo:crit}, \ref{theo:stable}, \ref{theo:stable_lattice} (see also Corollary~\ref{cor:tight}), the random variables $\{(Z_n-A_n)/B_n, n\in\N\}$ form a tight sequence (and even a convergent sequence, unless $c\in [c_1,c_2)$ and the distribution of $X$ is lattice). Thus, the first assumption of Lemma~\ref{lem:weak_comp} is satisfied.

To prove~\eqref{eq:plim_zn_an}, it suffices to show that the second assumption of Lemma~\ref{lem:weak_comp} is fulfilled, i.e.\
\begin{equation}\label{eq:lim_bn_an}
\lim_{n\to\infty} B_n/A_n=0.
\end{equation}

Assume first that $c\geq c_2$.
We have
$$
B_n=(\Var Z_n)^{1/2}
\leq N_n^{1/2}(\E[e^{2S_n}])^{1/2}
\sim  e^{((c+\varphi(2))/2)n},\;\;\; n\to\infty.
$$
Hence, to prove~\eqref{eq:lim_bn_an}, it suffices to show that $(c+\varphi(2))/2< c+\varphi(1)$. We have
$$
c\geq c_2=2\varphi'(2)-\varphi(2)>2(\varphi(2)-\varphi(1))-\varphi(2)=\varphi(2)-2\varphi(1),
$$
where we have used that $\varphi'(2)>\varphi(2)-\varphi(1)$ by the strict convexity of $\varphi$. It follows that
$$
2(c+\varphi(1))-(c+\varphi(2))=c+2\varphi(1)-\varphi(2) > 0,
$$
which proves~\eqref{eq:lim_bn_an} and verifies the second assumption of Lemma~\ref{lem:weak_comp}.

Assume now that $c\in(c_1,c_2)$. Then $A_n=\E Z_n$ and $B_n$ is defined by~\eqref{eq:def_Bn_large} or~\eqref{eq:def_bn_lattice}. In both cases,
$$
B_n=e^{b_n}=o(e^{\beta n}),\;\;\; n\to\infty.
$$
To prove~\eqref{eq:lim_bn_an}, it suffices to show that $\beta<c+\varphi(1)$. We have
$$
c+\varphi(1)-\beta=\alpha \varphi'(\alpha)-\varphi(\alpha)+\varphi(1)-\varphi'(\alpha)
=(\alpha-1) \varphi'(\alpha)-(\varphi(\alpha)-\varphi(1)).
$$
By convexity of $\varphi$,  and by the fact that $\alpha>1$ (which follows from $c>c_1$), the right-hand side is positive. This proves~\eqref{eq:lim_bn_an}.

Finally, assume that $c=c_1$. By~\eqref{eq:an_cases}, $A_n\sim (1/2) e^{\varphi'(1)n}$ as $n\to\infty$.
On the other hand, by~\eqref{eq:def_Bn_large} and~\eqref{eq:def_bn_lattice}, $B_n=o(e^{\beta n})=o(e^{\varphi'(1)n})$ as $n\to\infty$. This completes the proof of~\eqref{eq:lim_bn_an} and the proof of~\eqref{eq:plim_zn_an}.
\end{proof}

\subsection{Proof of Remark~\ref{rem:no_seq}}\label{sec:proof_rem}
The proof of Remark~\ref{rem:no_seq} will follow from the following standard lemma.
\begin{lemma}\label{lem:stand_lem}
Let $\{Y_n, n\in\N\}$ be a tight sequence of positive random variables,  and assume that there is no constant $x$ which is a weak accumulation point of $\{Y_n, n\in\N\}$. Then there is no sequence $a_n>0$ with
\begin{equation}\label{eq:Plim_Yn_an}
Y_n/a_n\toprobab 1,\;\;\; n\to\infty.
\end{equation}
\end{lemma}
\begin{proof}
Assume first that $\lim_{n\to\infty}a_n=\infty$. Then
$$
\lim_{n\to\infty}\P[Y_n/a_n>1/2]=\lim_{n\to\infty}\P[Y_n>a_n/2]=0,
$$
where the last equality follows from the tightness of the sequence $\{Y_n,n\in\N\}$. Hence, Eq.\ref{eq:Plim_Yn_an} cannot hold in this case.
Assume now that $\lim_{n\to\infty}a_n=0$. Then
$$
\liminf_{n\to\infty}\P[Y_n/a_n<2]=\liminf_{n\to\infty}\P[Y_n<2 a_n]<1,
$$
where the last equality follows from the assumption that $Y_n\nottoprobab 0$. Hence, Eq.\ref{eq:Plim_Yn_an} cannot hold in this case too.
Assume now that $\lim_{n\to\infty}a_n=a$, where $a>0$. Then Eq.\ref{eq:Plim_Yn_an} is not fulfilled since we have assumed that $Y_n\nottoprobab a$.

Now let the sequence  $a_n$ be arbitrary. Then it has a subsequence which  converges either to $\infty$, or to $0$, or to some limit $a>0$. Applying the above considerations, we see that~\eqref{eq:Plim_Yn_an} cannot hold.
\end{proof}

\begin{proof}[Proof of Remark~\ref{rem:no_seq}]
Let $B_n$ be defined as in  Theorem~\ref{theo:stable}, resp.\ as in Theorem~\ref{theo:stable_lattice}, in the non-lattice, resp.\ in the lattice case. Then by Theorems~\ref{theo:stable} and~\ref{theo:stable_lattice} applied in the case $c<c_1$, the sequence $Y_n:=Z_n/B_n$ satisfies the assumptions of Lemma~\ref{lem:stand_lem}. The statement of Remark~\ref{rem:no_seq} follows.
\end{proof}

\subsection{Proof of Theorem~\ref{theo:free_energy}}
First of all, note that the case $c\geq c_1$ follows from Theorem~\ref{theo:weak_lln}. Let us consider the case $c\in (0,c_1)$. Define $B_n$ as in~\eqref{eq:def_Bn_large}, resp.\ as in~\eqref{eq:def_bn_lattice}, in the non-lattice, resp. in the lattice case. Fix some $\eps>0$. We show that
\begin{equation}\label{eq:P_log_Zn_geq}
\lim_{n\to\infty}\P\left[\frac 1n \log Z_n>\beta+\eps\right]=0.
\end{equation}

By Theorem~\ref{theo:stable} and Theorem~\ref{theo:stable_lattice}, the family of random variables $\{Z_n/B_n, n\in\N\}$ is tight (and, in the non-lattice case, even convergent). Therefore, if $m_n$ is a sequence with $\lim_{n\to\infty}m_n=\infty$, then
\begin{equation}\label{eq:lim_P_Zn_Bn}
\lim_{n\to\infty}\P\left[\frac {Z_n}{B_n}>m_n\right]=0.
\end{equation}
We have
$$
\P\left[\frac 1n \log Z_n>\beta+\eps\right]=\P\left[Z_n>e^{(\beta+\eps)n}\right]=\P\left[\frac{Z_n}{B_n}>\frac{e^{(\beta+\eps)n}}{B_n}\right].
$$
Using~\eqref{eq:lim_P_Zn_Bn} and the fact that by~\eqref{eq:def_Bn_large} and~\eqref{eq:def_bn_lattice}, $B_n=o(e^{(\beta+\eps)n})$ as $n\to\infty$, we arrive at~\eqref{eq:P_log_Zn_geq}.

It remains to prove that
\begin{equation}\label{eq:P_log_Zn_leq}
\lim_{n\to\infty}\P\left[\frac 1n \log Z_n<\beta-\eps\right]=0.
\end{equation}
We have
$$
\P\left[\frac 1n \log Z_n<\beta-\eps\right]=\P\left[Z_n<e^{(\beta-\eps)n}\right]=\P\left[\frac{Z_n}{B_n}<\frac{e^{(\beta-\eps)n}}{B_n}\right].
$$
From the description of weak accumulation points of the sequence $\{Z_n/B_n, n\in\N\}$ given in  Theorem~\ref{theo:stable} and Theorem~\ref{theo:stable_lattice} it follows  that for every sequence $m_n$ such that $\lim_{n\to\infty}m_n=0$, we have
\begin{equation}\label{eq:lim_P_Zn_Bn1}
\lim_{n\to\infty}\P\left[\frac {Z_n}{B_n}<m_n\right]=0.
\end{equation}
Using this and the fact that by~\eqref{eq:def_Bn_large} and~\eqref{eq:def_bn_lattice}, $e^{(\beta-\eps)n}=o(B_n)$ as $n\to\infty$, we obtain~\eqref{eq:P_log_Zn_leq}.
This completes the proof.

\subsection{Proof of Proposition~\ref{prop:third_order}}
First of all note that $f_{+}$ is well-defined in a neighborhood of $c_1$. We show that $f_+(c_1)=f_-(c_1)$. Recall that $c_1=I(\varphi'(1))=\varphi'(1)-\varphi(1)$. Using this, we obtain
$$
f_+(c_1)=I^{-1}(c_1)=\varphi'(1)=\varphi(1)+c_1=f_-(c_1).
$$
Let us prove that $f_+'(c_1)=f_-'(c_1)$. By Lemma~\ref{lem:I_prime}, $I'(\varphi'(1))=1$. By the formula for the derivative of the inverse function,
$$
f_+'(c_1)=(I^{-1}(x))'|_{x=c_1}= \frac{1}{I'(I^{-1}(c_1))}=\frac{1}{I'(\varphi'(1))}=1=f_-'(c_1).
$$
Finally, let us show that $f''_+(c_1)\neq f''_-(c_1)$. To this end note that $f_-$ is linear, whereas $f_+$ is strictly concave as an inverse of a is strictly convex function $I$. This implies that  $f''_-(c_1)=0$ and $f''_+(c_1)<0$.

\section*{Acknowledgements}
The author is grateful to A.~Jan\ss en and M.~Schlather for several useful remarks.

\bibliographystyle{plainnat}
\bibliography{paper15bib}

\begin{thebibliography}{21}
\expandafter\ifx\csname natexlab\endcsname\relax\def\natexlab#1{#1}\fi
\expandafter\ifx\csname url\endcsname\relax
  \def\url#1{{\tt #1}}\fi

\bibitem[Bahadur and Ranga~Rao(1960)]{bahadur_rao60}
R.~Bahadur and R.~Ranga~Rao.
\newblock {On deviations of the sample mean.}
\newblock {\em Ann. Math. Stat.}, 31:\penalty0 1015--1027, 1960.

\bibitem[Ben~Arous et~al.(2005)Ben~Arous, Bogachev, and
  Molchanov]{ben_arous_etal05}
G.~Ben~Arous, L.~Bogachev, and S.~Molchanov.
\newblock {Limit theorems for sums of random exponentials.}
\newblock {\em Probab. Theory Relat. Fields}, 132\penalty0 (4):\penalty0
  579--612, 2005.

\bibitem[Bogachev(2006)]{bogachev06}
L.~Bogachev.
\newblock {Limit laws for norms of IID samples with Weibull tails.}
\newblock {\em J. Theor. Probab.}, 19\penalty0 (4):\penalty0 849--873, 2006.

\bibitem[Bogachev(2007)]{bogachev07}
L.~Bogachev.
\newblock {Extreme value theory for random exponentials.}
\newblock {Dawson, D. (ed.) et al., Probability and mathematical physics. A
  volume in honor of Stanislav Molchanov. CRM Proceedings and Lecture Notes 42,
  41-64 (2007). Providence, RI: AMS}, 2007.

\bibitem[Bovier(2006)]{bovier_book06}
A.~Bovier.
\newblock {\em {Statistical mechanics of disordered systems. A mathematical
  perspective.}}
\newblock {Cambridge Series in Statistical and Probabilistic Mathematics 18.
  Cambridge University Press}, 2006.

\bibitem[Bovier et~al.(2002)Bovier, Kurkova, and {L\"owe}]{bovier_etal02}
A.~Bovier, I.~Kurkova, and M.~{L\"owe}.
\newblock {Fluctuations of the free energy in the REM and the $p$-spin SK
  models.}
\newblock {\em Ann. Probab.}, 30\penalty0 (2):\penalty0 605--651, 2002.

\bibitem[{Cs\"org\"o} and {R\'ev\'esz}(1981)]{csorgo_book81}
M.~{Cs\"org\"o} and P.~{R\'ev\'esz}.
\newblock {\em {Strong approximations in probability and statistics.}}
\newblock {Probability and Mathematical Statistics. New York etc.: Academic
  Press}, 1981.

\bibitem[{Cs\"org\"o} et~al.(1986){Cs\"org\"o}, {Horv\'ath}, and
  Mason]{csorgo_etal86}
S.~{Cs\"org\"o}, L.~{Horv\'ath}, and D.~Mason.
\newblock {What portion of the sample makes a partial sum asymptotically stable
  or normal?}
\newblock {\em Probab. Theory Relat. Fields}, 72:\penalty0 1--16, 1986.

\bibitem[{Cs\"org\"o} and Mason(1986)]{csorgo_mason86}
S.~{Cs\"org\"o} and D.~Mason.
\newblock {The asymptotic distribution of sums of extreme values from a
  regularly varying distribution.}
\newblock {\em Ann. Probab.}, 14:\penalty0 974--983, 1986.

\bibitem[Dembo and Zeitouni(1993)]{dembo_book93}
A.~Dembo and O.~Zeitouni.
\newblock {\em {Large deviations techniques and applications.}}
\newblock {Boston, MA: Jones and Bartlett Publishers}, 1993.

\bibitem[Durrett(1979)]{durrett79}
R.~Durrett.
\newblock Maxima of branching random walks vs. independent random walks.
\newblock {\em Stochastic Process. Appl.}, 9\penalty0 (2):\penalty0 117--135,
  1979.

\bibitem[Durrett(1991)]{durrett_book91}
R.~Durrett.
\newblock {\em {Probability. Theory and examples.}}
\newblock {Brooks/Cole Statistics/Probability Series}, 1991.

\bibitem[Eisele(1983)]{eisele83}
Th. Eisele.
\newblock {On a third-order phase transition.}
\newblock {\em Commun. Math. Phys.}, 90:\penalty0 125--159, 1983.

\bibitem[Galves et~al.(1989)Galves, Martinez, and Picco]{galves_etal89}
A.~Galves, S.~Martinez, and P.~Picco.
\newblock {Fluctuations in Derrida's random energy and generalized random
  energy models}.
\newblock {\em J. Stat. Phys.}, 54:\penalty0 515--529, 1989.

\bibitem[Gnedenko and Kolmogorov(1954)]{gnedenko_book54}
B.V. Gnedenko and A.N. Kolmogorov.
\newblock {\em {Limit distributions for sums of independent random variables.}}
\newblock {Cambridge: Addison-Wesley Publishing Company}, 1954.

\bibitem[Ivchenko(1973)]{ivchenko73}
G.~Ivchenko.
\newblock {Variational series for a scheme of summing independent variables.}
\newblock {\em Theory of Probab. Appl.}, 18:\penalty0 531--545, 1973.

\bibitem[{Jan\ss en}(2009)]{janssen09}
A.~{Jan\ss en}.
\newblock {Limit laws for power sums and norms of i.i.d. samples.}
\newblock {\em Probab. Theory Related Fields}, To appear. DOI:
  10.1007/s00440-008-0198-y, 2009.

\bibitem[Olivieri and Picco(1984)]{olivieri_picco84}
E.~Olivieri and P.~Picco.
\newblock {On the existence of thermodynamics for the random energy model.}
\newblock {\em Commun. Math. Phys.}, 96:\penalty0 125--144, 1984.

\bibitem[Petrov(1965)]{petrov65}
V.~Petrov.
\newblock {On the probabilities of large deviations for sums of independent
  random variables.}
\newblock {\em Theor. Probab. Appl.}, 10:\penalty0 287--298, 1965.

\bibitem[Petrov(1995)]{petrov_book95}
V.~Petrov.
\newblock {\em {Limit theorems of probability theory. Sequences of independent
  random variables.}}
\newblock {Oxford Studies in Probability, 4. Oxford: Clarendon Press}, 1995.

\bibitem[Schlather(2001)]{schlather01}
M.~Schlather.
\newblock {Limit distributions of norms of vectors of positive i.i.d.\ random
  variables.}
\newblock {\em Ann. Probab.}, 29\penalty0 (2):\penalty0 862--881, 2001.

\end{thebibliography}
\end{document}